\documentclass[a4paper,12pt]{article}

\usepackage{bm}
\usepackage{cancel}
\usepackage{amsmath}
\usepackage{amssymb}
\usepackage{euler,eucal}
\usepackage{amsfonts}
\usepackage{amsthm}

\usepackage[top=2.8cm,left=2.5cm,right=2.2cm,bottom=3cm]{geometry}

\usepackage[utf8]{inputenc}
\usepackage[T1]{fontenc}
\usepackage[english]{babel}

\usepackage{algorithmic}
\usepackage[ruled]{algorithm}

\usepackage[unicode=true,plainpages=false]{hyperref}
\hypersetup{colorlinks=false,pdfborder={0 0 0.1},linkcolor=magenta,anchorcolor=magenta,urlcolor=blue,citecolor=blue,
          pdftitle={Polynomial Chaos Expansion of random coefficients and the solution of stochastic partial differential equations in the Tensor Train format},
          pdfauthor={Sergey Dolgov, Boris N. Khoromskij, Alexander Litvinenko and Hermann G. Matthies}
          }

\DeclareMathAlphabet{\mathitbf}{OML}{cmm}{b}{it}
\def\jmath{j}

\newcommand{\dotprod}[2]{ \left \langle #1, #2 \right \rangle }

\def\eps{\varepsilon}

\def\nin{\hspace{0.22em}/ \hspace{-0.88em} \in}       

\def\bSigma{\bm{\Sigma}}
\def\u{\bm{u}}
\def\U{\bm{U}}
\def\uh{\mathbf{u}}
\def\f{\mathbf{f}}
\def\K{\mathbf{K}}
\def\BDelta{\bm{\Delta}}

\def\I{\mathbb{I}}
\def\J{\mathbb{J}}

\def\Js{\mathcal{J}}

\newcommand{\cov}{\mbox{cov}}
\newcommand{\var}{\mbox{var}}

\def\thetab{\bm{\theta}}
\def\alphab{\bm{\alpha}}
\def\betab{\bm{\beta}}
\def\pb{\bm{p}}

\newtheorem{theorem}{Theorem}

\newtheorem{lemma}[theorem]{Lemma}
\newtheorem{definition}[theorem]{Definition}
\newtheorem{statement}[theorem]{Statement}
\newtheorem{remark}[theorem]{Remark}

\date{March 11, 2015}

\begin{document}

\title{Polynomial Chaos Expansion of random coefficients and the solution of stochastic partial differential equations in the Tensor Train format}

\author{Sergey Dolgov\footnotemark[2] \and Boris N. Khoromskij\footnotemark[3] \and Alexander Litvinenko\footnotemark[4] \and Hermann G. Matthies\footnotemark[5]}



%
\maketitle

\thispagestyle{plain}

\renewcommand{\thefootnote}{\fnsymbol{footnote}}

\footnotetext[2]{E-mail: {\tt dolgov@mpi-magdeburg.mpg.de}. Max-Planck-Institut f\"ur Dynamik komplexer technischer Systeme, Sandtorstr. 1, 39106 Magdeburg, Germany.}
\footnotetext[3]{E-mail: {\tt bokh@mis.mpg.de}. Max-Planck-Institut f\"ur Mathematik in den Naturwissenschaften,  Inselstra\ss e 22, 04103 Leipzig, Germany.}
\footnotetext[4]{E-mail: {\tt alexander.litvinenko@kaust.edu.sa}. King Abdullah University of Science and Technology (KAUST), Thuwal, Saudi Arabia.}
\footnotetext[5]{E-mail: {\tt wire@tu-bs.de}. Institute for Scientific Computing, Technische Universit\"at Braunschweig, Hans-Sommerstr. 65, Brunswick, Germany.}

\renewcommand{\thefootnote}{\arabic{footnote}}

\begin{center}
\textbf{\small{
This is a major revision of the manuscript \url{http://arxiv.org/abs/1406.2816}.
We have significantly extended the numerical experiments, adding
the comparison of cross algorithms,
verification via the Monte Carlo method,
computation of the exceedance probabilities,
the log-normally distributed coefficient,
and the systematic study of the performance w.r.t. different parameters of the stochastic PDE (correlation length, variance, etc.).
Some unused material is removed.
}}
\end{center}

\begin{abstract}
We apply the Tensor Train (TT) decomposition to construct the tensor product Polynomial Chaos Expansion (PCE) of a random field, to solve the stochastic elliptic diffusion PDE with the stochastic Galerkin discretization, and to compute some quantities of interest (mean, variance, exceedance probabilities).
We assume that the random diffusion coefficient is given as a smooth transformation of a Gaussian random field.
In this case, the PCE is delivered by a complicated formula, which lacks an analytic TT representation.
To construct its TT approximation numerically, we develop the new block TT cross algorithm, a method that computes the whole TT decomposition from a few evaluations of the PCE formula.
The new method is conceptually similar to the adaptive cross approximation in the TT format, but is more efficient when several tensors must be stored in the same TT representation, which is the case for the PCE.
Besides, we demonstrate how to assemble the stochastic Galerkin matrix and to compute the solution of the elliptic equation and its post-processing, staying in the TT format.

We compare our technique with the traditional sparse polynomial chaos and the Monte Carlo approaches.
In the tensor product polynomial chaos, the polynomial degree is bounded for each random variable independently.
This provides higher accuracy than the sparse polynomial set or the Monte Carlo method, but the cardinality of the tensor product set grows exponentially with the number of random variables.
However, when the PCE coefficients are implicitly approximated in the TT format, the computations with the full tensor product polynomial set become possible.
In the numerical experiments, we confirm that the new methodology is competitive in a wide range of parameters, especially where high accuracy and high polynomial degrees are required.
\end{abstract}

{\par \it Keywords:~}{uncertainty quantification, polynomial chaos expansion, Karhunen-Lo\`{e}ve expansion, stochastic Galerkin, tensor product methods, tensor train format, adaptive cross approximation, block cross}

{\par \it MSC2010:~}{
15A69,  
65F10,  
60H15,  
60H35,  
65C30  
}

\section{Motivation}
\label{sec:1}

Situations in which one is concerned with uncertainty quantification often come in the following guise: we are investigating physical models where inputs are not given precisely, but instead are random quantities or random fields, or depend on a set of parameters.
A classical example is the Darcy flow model with a random diffusion coefficient,
\begin{equation}
\label{eq:elliptic}
-\nabla \kappa(x,\omega) \nabla u(x,\omega) = f(x,\omega), \quad x \in D\subset \mathbb{R}^d,
\end{equation}
where $d$ is the spatial dimension, $\kappa(x,\omega)$ is a random field dependent on a random
parameter $\omega$ in a probability space $(\Omega,\mathcal{A},\mathcal{P})$.
The solution $u(x,\omega)$ belongs to $H^1(D)$ w.r.t. $x$ and the same probability space w.r.t. $\omega$.
There is an established theory about the existence and uniqueness of the solution to \eqref{eq:elliptic} under various assumptions on $\kappa$ and $f$; see, for example, 
\cite{babuska2004galerkin,Sarkis09,GITTELSON10,matthiesKeese05cmame,mugler2011elliptic}.

In \cite{Sarkis09, GITTELSON10} it is shown that under additional
assumptions on the right-hand side $f$ and special choices of the test space the problem \eqref{eq:elliptic} is well-posed.
The case where the Lax-Milgram theorem is not applicable (e. g. upper and lower constants $\underline{\kappa}$, $\overline{\kappa}$ in
$0<\underline{\kappa}<\kappa<\overline{\kappa}<\infty$ do not exist)
is also considered in \cite{mugler2011elliptic}.
In \cite{Sarkis09} the authors analyze assumptions on $\kappa$ which were made in \cite{babuska2004galerkin} to guarantee the uniqueness and the existence of the solution and to
offer a new method with much weakened assumptions. Additionally, in \cite{Sarkis09}, the authors point out that after truncating terms in the expansion for $\kappa$,
as is done in \cite{matthiesKeese05cmame}, it is not guaranteed that the truncated $\kappa$ will be strictly bounded from zero. As a result the existence of
the approximate solution to \eqref{eq:elliptic} is questionable. The remarkable difference of the approach in \cite{Sarkis09} from approaches in
\cite{babuska2004galerkin, matthiesKeese05cmame} is that the permeability coefficient $\kappa$ is not truncated and, as a result, the ellipticity condition is maintained.

To solve \eqref{eq:elliptic} we need to discretize the random field  $\kappa$.
This requires some knowledge on the probability space $\Omega$ and the probability measure.
A widely used approach relies on two assumptions: $\kappa$ is defined as an invertible smooth function of another random field with known distribution (e.g. normal), and the covariance function of either of those fields is given.
After that $\kappa(x,\omega)$ can be \emph{expanded} to a series of functions, depending on $x$ and a set of new parameters $\thetab=(\theta_1,\theta_2,\ldots)$.
Typically used are polynomials of $\thetab$, hence called the \emph{Polynomial Chaos Expansion} (PCE) \cite{wiener38, xiuKarniadakis02a} family of discretizations.
Another approach is the collocation on some grid in $\thetab$  \cite{Babushka_Colloc,back2011stochastic}.
Each of $\thetab$ is a random quantity depending on $\omega$, but the domain of $\thetab$ is known deterministically.
Introducing a discretization for $\thetab$, we turn the stochastic problem \eqref{eq:elliptic} into a high-dimensional deterministic one.
However, its straightforward numerical solution suffers from the \emph{curse of dimensionality}: even if $\thetab$ contains a finite amount of variables (say, $M$), the number of discrete degrees of freedom grows exponentially with $M$.

To surmount this issue, a data-sparse approximation is needed.
During the last years, low-rank tensor product techniques were successfully applied to the solution of high-dimensional stochastic and parametric PDEs.
A recent literature survey of low-rank tensor approximation techniques can be found in \cite{Grasedyck_Kressner}.
The tensor product approach involves a \emph{format}, in which the data are represented, and a set of \emph{algorithms} to manipulate the data in the format.
The algorithms can be roughly separated into three classes:
  methods performing basic approximate algebra (additions, products, simple iterative methods) of data within the format;
  methods constructing the formatted representation from a few entries of the initial data;
  and methods aiming to solve equations, e.g. linear systems, ODEs or eigenvalue problems, keeping all data in the format.

To some extent, these methods have already been applied to parametric problems.
Non-intrusive (black box) tensor methods for multi-parametric problems, i.e. ``class 2'', were developed in \cite{Ballan_Grasedyck14, Ballan_Grasedyck10, doostan-non-intrusive-2013}.
In particular, in \cite{Ballan_Grasedyck14} the authors follow the stochastic collocation approach and compute functionals of the solution of multi-parametric PDEs.
Since the stochastic collocation allows to solve uncoupled deterministic problems for different collocation points, the functional of the solution (e.g. the average value) can be approximated straightforwardly via the black box hierarchical tensor interpolation algorithm.
To compute the whole stochastic solution is a more difficult problem, especially in the stochastic Galerkin framework, where deterministic problems are coupled.

In \cite{KressTobler11, KressTobler11HighDim, matthieszander-lowrank-2012, Sousedik2014truncated, ZanderDiss} the authors develop iterative methods and preconditioners to solve numerically discretized multi-parametric problems.
Several manipulations of the PCE with a low-rank approximation have been considered.
In \cite{litvinenko-spde-2013} the authors assume that the solution has a low-rank \emph{canonical} (CP) tensor format and develop methods for the CP-formatted computation of level sets.
In \cite{UQLitvinenko12,wahnert-stochgalerkin-2014} the authors analyzed tensor ranks of the stochastic operator.
The proper generalized decomposition was applied for solving high dimensional stochastic problems in \cite{Nouy07,nouy-pgd-stoch-2010}.
In \cite{khos-pde-2010, khos-qtt-2010, KhSch-Galerkin-SPDE-2011} the authors employed newer tensor formats, the Tensor Train (TT) and Quantized TT (QTT), for the
approximation of coefficients and the solution of stochastic elliptic PDEs.
The theoretical study of the complexity of the stochastic equation was provided, for example, by means of the analytic regularity and (generalized) PC approximation \cite{xiuKarniadakis02a} for control problems constrained by linear parametric elliptic and parabolic PDEs \cite{kunoth-2013}.

Other classical techniques to cope with high-dimensional problems are sparse grids
\cite{Griebel,Griebel_Bungartz, nobile2008sparse}
and (quasi) Monte Carlo methods \cite{graham-QMC-2011,scheichl-mlmc-further-2013,Schwab-MLQMC-2015}.
Nevertheless, tensor product methods are more flexible than sparse grids, as they allow to avoid severe reductions of the model from the very beginning, and
adapt a suitable structure on the discrete level.
Compared to Monte Carlo methods, tensor techniques work implicitly with the whole solution, and even a construction of a tensor format for entry-wise given data in a black box manner uses less randomness than the MC approach.

In this article we approximate the PCE of the input coefficient $\kappa(x,\omega)$ in the TT tensor format. After that we compute the solution $u(x,\omega)$ and perform all post-processing in the same TT format.
The first stage, computation of the PCE of $\kappa$, involves a lengthy formula, defining each entry of the discretized coefficient.
To perform this computation efficiently, we develop a new \emph{block cross} approximation algorithm, which constructs the TT format for $\kappa$ from a few evaluations of the entry-wise formula.
This formula delivers several tensors that are to be summed and approximated in a TT format.
We show that the new algorithm is more efficient than several runs of a previously existing cross method \cite{so-dmrgi-2011proc} for each tensor separately.
As soon as the coefficient is given in the TT format, it becomes very easy to construct the stiffness matrix, derived from the stochastic Galerkin discretization of \eqref{eq:elliptic}.
We apply the alternating iterative tensor algorithm to solve a large linear system arising from \eqref{eq:elliptic}, and finally use the cross algorithm again to compute the exceedance probability from the solution.

In the next section, we outline the general Galerkin, Polynomial Chaos Expansion (PCE) and Karhunen-Lo\`{e}ve Expansion (KLE) discretization schemes for a random field.
An introduction to the TT methods and the new block cross interpolation algorithm are presented in Section \ref{sec:TTlow}.
Some details of how to apply the block cross algorithm to the PCE calculations are given in Section \ref{sec:tt-kle-pce}.
We start with the TT approximation of the multi-dimensional input coefficient $\kappa$.
After that, in Section \ref{sec:stiff_tt}, we construct the stochastic Galerkin matrix in the TT format.
Section \ref{subsec:levelSets} is devoted to the efficient post-processing (computation of the mean value, covariance, and probability of a particular event) in the TT format.
Numerical results in Section \ref{sec:numerics} demonstrate the practical performance of the TT approach in the outlined tasks.

\section{Discretisation and computation}
\label{SS:discret}
For brevity, we follow \cite{Matthies_encycl}, where more references may be found.
See also the recent monograph \cite{LeMatre10} to study more about KLE, PCE and multiindices.
In \cite{ ErnstU10, Sousedik2014truncated, Ullmann10, ullmann2012efficient} the authors discuss the stochastic Galerkin matrix, its sparsity and preconditioners.

To discretize \eqref{eq:elliptic}, we follow the Galerkin approach.
The Finite Element method (for example, with piecewise linear functions) is a natural way to discretize the spatial part.
We choose a finite dimensional subspace
\begin{equation}
\mathcal{V}_N = \mathrm{span}\{\varphi_1(x),\ldots,\varphi_N(x)\} \subset \mathcal{V},
\label{eq:V_N}
\end{equation}
where $\mathcal{V} = H^1(D) \cap L_2^0(\bar D)$ is the Hilbert space of functions on $x$.
For simplicity, we impose the homogeneous Dirichlet boundary conditions.

Discretization in the probability space $(\Omega, \mathcal{A},\mathcal{P})$ is less trivial.
We use the Karhunen-Lo\`{e}ve Expansion (KLE) to determine a finite set of independent parameters, defining $\kappa(x,\omega)$.
However, the distribution of these parameters might be unknown, and special efforts are needed to resolve this.

\subsection{Discretization of the input random field}
We assume that $\kappa(x,\omega)$ may be seen as a smooth transformation $\kappa = \phi(\gamma)$ of the Gaussian random field $\gamma(x,\omega)$.
In this Section, we explain how to compute KLE of $\gamma$ if the covariance of $\kappa$ is given.
For more details see \cite[Sections 3.4.1, 3.4.2]{ZanderDiss} or \cite{KeeseDiss}.

A typical example is the log-normal field with $\phi(\gamma) = \exp(\gamma)$.
Expanding $\phi$ in a series of the Hermite polynomials gives
\begin{equation}
\phi(\gamma) = \sum\limits_{i=0}^{\infty} \phi_i h_i(\gamma) \approx \sum\limits_{i=0}^{Q} \phi_i h_i(\gamma), \quad \phi_i = \int\limits_{-\infty}^{+\infty} \phi(z) \frac{1}{i!} h_i(z) \exp(-z^2/2) dz,
\label{eq:phi_tranform}
\end{equation}
where $h_i(\cdot)$ is the $i$-th Hermite polynomial, and $Q$ is the number of terms after the truncation.

The Gaussian field $\gamma(x,\omega)$ can be written as the Karhunen-Lo\`{e}ve expansion.
First, given the covariance function of $\kappa(x,\omega)$, we may relate it with the covariance function of $\gamma(x,\omega)$ as follows (see details in \cite[Sections 3.4.1, 3.4.2]{ZanderDiss}),
\begin{equation}
\cov_{\kappa}(x,y) = \int_{\Omega} \left(\kappa(x,\omega)-\bar{\kappa}(x)\right) \left(\kappa(y,\omega)-\bar{\kappa}(y)\right) dP(\omega) \approx \sum_{i=0}^{Q} i! 
\phi_i^2 \cov_{\gamma}^i (x,y),
\label{eq:cov_transform}
\end{equation}
where $\bar{\kappa}(x)$ is the expectation of $\kappa(x,\omega)$.
Solving this implicit $Q$-order equation, we derive $\cov_{\gamma} (x,y)$ \cite{ZanderDiss}.
Now, the KLE can be computed as follows:
\begin{equation}
\gamma(x,\omega) = \sum_{m=1}^{\infty} g_m(x) \theta_m(\omega), \quad \mbox{where} \quad \int\limits_{D} \cov_{\gamma}(x,y) g_m(y)dy = \lambda_m g_m(x).
\end{equation}
Here we assume that the eigenfunctions $g_m$ absorb the square roots of the KL eigenvalues. The stochastic variables $\theta_m$ are normalized (they are uncorrelated and jointly Gaussian).

The initial coefficient $\kappa$ depends nonlinearly on $\theta_m$.
In the discrete setting, we truncate PCE and write it for $M$ random variables,
\begin{equation}
\kappa(x,\omega) \approx \sum\limits_{\alphab \in \Js_M} \kappa(x,\alphab) H_{\alphab}(\thetab(\omega)), \quad \text{where}\quad H_{\alphab}(\thetab):=h_{\alpha_1}(\theta_1) \cdots h_{\alpha_M}(\theta_M)
\label{eq:pce_k}
\end{equation}
is the multivariate Hermite polynomial, $\alphab = (\alpha_1,\ldots,\alpha_M)$ is a multiindex (a tuple of multinomial orders), $h_{\alpha_m}(\theta_m)$ is the univariate Hermite polynomial, $\thetab = (\theta_1,\ldots,\theta_M)$ is a tuple of random variables,
and $\Js_M$ is a set of all multinomial orders (see definition below).
The Galerkin coefficients $\kappa(x,\alphab)$ are evaluated as follows:
\begin{equation}
\kappa(x,\alphab) = \dfrac{(\alpha_1+\cdots+\alpha_M)!}{\alpha_1! \cdots \alpha_M! } \phi_{\alpha_1+\cdots+\alpha_M} \prod_{m=1}^{M} g_m^{\alpha_m}(x),
\label{eq:pce_k_coeffs}
\end{equation}
where $\phi_{\alpha_1+\cdots+\alpha_M}$ is the Galerkin coefficient of the transform function in \eqref{eq:phi_tranform}
and $g_m^{\alpha_m}(x)$ means just the $\alpha_m$-th power of the KLE function value $g_m(x)$.

In practice, we restrict the polynomial orders in \eqref{eq:pce_k} to finite limits, which can be done in different ways.

\begin{definition}
\label{def:tensor_product}
The \emph{full} multi-index set is defined by restricting each component independently,
$$
\Js_{M,\pb} = \left\{\alphab=(\alpha_1,\ldots,\alpha_M):~~\alpha_m = 0,\ldots,p_m,~~m=1,\ldots,M \right\},
$$
where $\pb=(p_1,\ldots,p_M)$ is a shortcut for the tuple of order limits.
\end{definition}

The full set provides high flexibility for the resolution of stochastic variables 
\cite{wahnert-stochgalerkin-2014, litvinenko-spde-2013, Sousedik2014truncated}.
However, its cardinality is equal to $\prod_{m=1}^M (p_m+1) \le (p+1)^M$, if $p_m\le p$.
This may become a enormous number even if $p$ and $M$ are moderate ($p \sim 3$, $M \sim 20$ is typical).
In this paper, we do not store all $(p+1)^M$ values explicitly, but instead approximate them via a \emph{tensor product} representation.

Another approach is to preselect the set of polynomials with the moderate cardinality \cite{TodorSCA, matthiesKeese05cmame, Sudret_sparsePCE, chkifa-adapt-stochfem-2015}.
\begin{definition}
The \emph{sparse} multi-index set is defined by restricting the sum of components,
$$
\Js_{M,p}^{sp} = \left\{\alphab=(\alpha_1,\ldots,\alpha_M):\quad \alphab \ge 0,~~\alpha_1+\cdots+\alpha_M \le p \right\}.
$$
\end{definition}

The sparse set contains $\mathcal{O}\left(\frac{M!}{p!(M-p)!}\right) = \mathcal{O}(M^p)$ values if $M \gg p$,
which is definitely less than $(p+1)^M$.
However, the negative side is that for a fixed $p$ some variables are resolved worse than others, and the approximation accuracy may suffer.
It may also be harmful to increase $p$ since in the sparse set it contributes exponentially to the complexity.

The low-rank tensor approximation allows to reduce the storage of the coefficient with the full set $\Js_{M,\pb}$ from $\mathcal{O}(p^M)$ to $\mathcal{O}(M p r^2)$, where $r$ is a data-dependent \emph{tensor rank}.
Theoretical estimates of $r$ are under development; numerical studies reflect that often $r$ does not depend on $p$ and depends linearly (or even milder) on $M$ \cite{khos-pde-2010}.
When $M$ is moderate, and $p$ is relatively large, the low-rank approach with the full index set (Def. \ref{def:tensor_product}) becomes preferable.
Besides, as soon as the coefficient $\kappa$ is computed in a tensor product form, to assemble the Galerkin matrix of the stochastic PDE is a much easier task than with the sparse set.

Some complexity reduction in Formula \eqref{eq:pce_k_coeffs} can be achieved with the help of the KLE for the initial field $\kappa(x,\omega)$.
Consider the expansion
\begin{equation}
\kappa(x,\omega) = \bar\kappa(x) + \sum_{\ell=1}^{\infty} \sqrt{\mu_\ell} v_\ell(x) \eta_\ell(\omega)
\approx  \bar\kappa(x) + \sum_{\ell=1}^{L} \sqrt{\mu_\ell} v_\ell(x) \eta_\ell(\omega),
\label{eq:k_kle}
\end{equation}
where $v_\ell(x)$ are eigenfunctions of the integral operator with the covariance as the kernel (see e.g. \cite{ZanderDiss, KeeseDiss, matthiesKeese05cmame}).
It's hard to work with \eqref{eq:k_kle} straightforwardly, since the distribution of $\eta_\ell$ is generally unknown.
But we know that the set $V(x) = \{v_\ell(x)\}_{\ell=1}^{L}$, where $L$ is the number of KLE terms after the truncation, serves as an optimally reduced basis.
Therefore, instead of using \eqref{eq:pce_k_coeffs} directly, we project it onto $V(x)$:
\begin{equation}
\bm{\tilde \kappa}_{\ell}(\alphab) = \dfrac{(\alpha_1+\cdots+\alpha_M)!}{\alpha_1! \cdots \alpha_M! } \phi_{\alpha_1+\cdots+\alpha_M} \int\limits_{D} \prod_{m=1}^{M} g_m^{\alpha_m}(x) v_\ell(x) dx.
\label{eq:k_kle_pce_coeff}
\end{equation}
Note that the range $\ell=1,\ldots,L$ may be much smaller than the number of discretization points $N$ in $D$.
After that, we restore the approximate coefficients \eqref{eq:pce_k_coeffs}:
\begin{equation}
\kappa(x,\alphab) \approx \bar\kappa(x) + \sum\limits_{\ell=1}^{L} v_\ell(x) \bm{\tilde \kappa}_{\ell}(\alphab).
\label{eq:k_kle_pce}
\end{equation}

\subsection{Discretization of the stochastic operator equation}
The same PCE ansatz of the coefficient \eqref{eq:pce_k} may be adopted to discretize the solution $u$ and ultimately the whole initial problem \eqref{eq:elliptic},
see \cite{wahnert-stochgalerkin-2014, litvinenko-spde-2013}.
For brevity, we illustrate the procedure on the deterministic right-hand side $f=f(x)$.

Given the KLE components \eqref{eq:k_kle} and the spatial discretization basis \eqref{eq:V_N},
we first assemble the spatial Galerkin matrices,
\begin{equation}
\bm{K}_0^{(x)}(i,j) = \int\limits_{D} \bar\kappa(x) \nabla \varphi_i(x) \cdot \nabla \varphi_j(x) dx, \quad \bm{K}_\ell^{(x)}(i,j) = \int\limits_{D} v_\ell(x) \nabla \varphi_i(x) \cdot \nabla \varphi_j(x) dx,
\label{eq:K_l}
\end{equation}
for $i,j=1,\ldots,N$, $\ell=1,\ldots,L$.
Now we take into account the PCE part $\bm{\tilde \kappa}_{\alphab}$.
Assuming that $u$ is decomposed in the same way as \eqref{eq:pce_k} with the same $\Js_{M,\pb}$ or $\Js_{M,p}^{sp}$, we compute the integral in \eqref{eq:K_p} over stochastic coordinates $\thetab$ and 
compute the stochastic parts $\bm{K}_{\ell}^{(\omega)}\in \mathbb{R}^{\#\Js_M \times \# \Js_M}$ of the Galerkin matrix as follows (see also \cite{KeeseDiss, ZanderDiss, Matthies_encycl}),
\begin{equation}
\bm{K}_{\ell}^{(\omega)}(\alphab,\betab) = \int\limits_{\mathbb{R}^M} H_{\alphab}(\thetab) H_{\betab}(\thetab) \sum\limits_{\bm{\nu} \in \Js_{M}} 
\bm{\tilde\kappa}_{\ell}(\bm{\nu}) H_{\nu}(\theta) \rho(\thetab)d\thetab = \sum\limits_{\nu \in \Js_{M}} \BDelta_{\alphab,\betab,\bm{\nu}} \bm{\tilde\kappa}_{\ell}(\bm{\nu}),
\label{eq:K_p}
\end{equation}
where $\rho(\thetab) = \rho(\theta_1) \cdots \rho(\theta_M)$ is the probability density with $\rho(\theta_m) = \frac{1}{\sqrt{2\pi}} \exp(-\theta_m^2/2)$, and
\begin{equation}
\BDelta_{\alphab,\betab,\bm{\nu}} =  \Delta_{\alpha_1,\beta_1,\nu_1} \cdots \Delta_{\alpha_M,\beta_M,\nu_M}, \quad  \Delta_{\alpha_m,\beta_m,\nu_m} = 
\int\limits_{\mathbb{R}} h_{\alpha_m}(\theta) h_{\beta_m}(\theta) h_{\nu_m}(\theta) \rho(\theta) d\theta,
\label{eq:Delta}
\end{equation}
is the triple product of the Hermite polynomials, and $\bm{\tilde\kappa}_{\ell}(\bm{\nu})$ is according to \eqref{eq:k_kle_pce_coeff}.
Let us denote $\BDelta_0(\alphab,\betab) = \Delta_{\alpha_1,\beta_1,0} \cdots \Delta_{\alpha_M,\beta_M,0}$, i.e. $\BDelta_0 \in \mathbb{R}^{\#J_M \times \# J_M}$.
Putting together \eqref{eq:k_kle_pce}, \eqref{eq:K_l} and \eqref{eq:K_p}, we obtain the whole discrete stochastic Galerkin matrix,
\begin{equation}
\K = \bm{K}_0^{(x)} \otimes \BDelta_0 + \sum\limits_{\ell=1}^L \bm{K}_\ell^{(x)} \otimes \bm{K}_{\ell}^{(\omega)},
\label{eq:K_gen}
\end{equation}
which is a square matrix of size $N \cdot \#\Js_{M}$, and $\otimes$ is the Kronecker product.

For the sparse index set, we need to compute $\mathcal{O}\left((\#\Js_{M,p}^{sp})^3\right)$ entries of $\BDelta$ explicitly.
For the full index set and $\bm{\tilde \kappa}_{\bm{\nu}}$ given in the tensor format, the direct product in $\BDelta$ \eqref{eq:Delta} allows to exploit the same format for \eqref{eq:K_gen} and to simplify the procedure, see Sec. \ref{sec:stiff_tt}.

The deterministic right-hand side is extended to the size of $\K$ easily,
\begin{equation}
\f = \bm{f}_0 \otimes \bm{e}_{0}, \qquad \bm{f}_0(i) = \int\limits_{D} \varphi_i(x) f(x)dx, \quad i=1,\ldots,N,
\label{eq:f}
\end{equation}
and $\bm{e}_0$ is the first identity vector of size $\#\Js_{M}$, $\bm{e}_0 = (1,0,\ldots,0)^\top$, which assigns the deterministic $f(x)$ to the zeroth-order Hermite polynomial in the parametric space.

\subsection{Solution of the stochastic equation}
Now the discrete equation writes as a linear system
\begin{equation}
\K \uh = \f,
\label{eq:linsystem}
\end{equation}
where $\uh\in\mathbb{R}^{N \cdot \#\Js_{M}}$ is the vector of the Galerkin coefficients for the solution, with elements enumerated by $u(i,\alphab)$.

In \cite{Sousedik2014truncated} the authors propose two new strategies for constructing preconditioners for the stochastic Galerkin system to be used with Krylov subspace
iterative solvers. The authors also research the block sparsity structure of the corresponding coefficient matrix as well as the stochastic Galerkin matrix.
In \cite{Ullmann10, ullmann2012efficient} the authors develop and analyze a Kronecker-product preconditioner.

In many cases it is sufficient to use the mean-field preconditioner $\bm{K}_0^{(x)} \otimes \BDelta_0$, which is easy to invert due to the Kronecker form.
We follow this approach.
To solve the system in a tensor product format, we employ alternating optimization methods \cite{ds-amen-2014}.

\section{Tensor product formats and low-rank data compression}
\label{sec:TTlow}

We see that the cardinality of the full polynomial set $\Js_{M,\pb}$ may rise to prohibitively large values $(p+1)^M$.
In this paper, we study two ways to alleviate this problem.
First, we can fix the basis set a priori.
This is the case with the sparse set $\Js_{M,p}^{sp}$.
Due to particular properties of the stochastic elliptic equation, it is possible to derive a posteriori error indicators and build the sparse set adaptively \cite{eigel-adapt-stoch-fem-2014,chkifa-adapt-stochfem-2015}.

Another approach, which is applicable to a wider class of problems, is to use the full discretization set, but to approximate already discrete data, via a low-rank tensor product format.
For stochastic PDEs, low-rank approximations were used in e.g. \cite{beylkin-high-2005,doostan-non-intrusive-2013,litvinenko-spde-2013,khos-pde-2010, matthieszander-lowrank-2012, nouy-pgd-stoch-2010}.
This approach requires specific computational algorithms, since the data are represented in a nonlinear fashion.
In this section we suggest such an algorithm to construct a data-sparse format of the stochastic coefficient and quantities of interest.

\subsection{Tensor Train decomposition}
To show the techniques in the briefest way, we choose the so-called \emph{matrix product states} (MPS) formalism~\cite{fannes-mps-1992}, which introduces the following representation of a multi-variate array, or \emph{tensor}:
\begin{equation}\label{eq:tt}
  \uh(\alpha_1,\ldots,\alpha_M)  = \sum_{s_1=1}^{r_1} \sum_{s_2=1}^{r_2}  \cdots \sum_{s_{M-1}=1}^{r_{M-1}}
      \u^{(1)}_{s_0,s_1}(\alpha_1) \u^{(2)}_{s_1,s_2}(\alpha_2) \cdots \u^{(M)}_{s_{M-1},s_M}(\alpha_M).
\end{equation}
Surely, in the same form we may write $\bm{\kappa}(\alphab)$.
In numerical linear algebra this format is known as~``\emph{tensor train} (TT)'' representation \cite{osel-tt-2011, ot-tt-2009}.
Each TT core (or \emph{block}) is a three-dimensional array, $\u^{(m)} \in \mathbb{R}^{r_{m-1} \times (p_m+1) \times r_m}$, $m=1,\ldots,M$, where
$p_m$ denotes the \emph{mode size}, the polynomial order in the variable $\alpha_m$, and $r_m=r_m(\uh)$ is the \emph{TT rank}.
The total number of entries scales as $\mathcal{O}(Mpr^2)$, which is tractable as long as $r=\max\{r_m\}$ is moderate.

We still have not specified the \emph{border} rank indices $s_0$ and $s_M$.
In the classical TT definition \cite{osel-tt-2011} and in \eqref{eq:tt}, there is only one tensor $\uh(\alphab)$ in the left-hand side, therefore $s_0=s_M=1$, and also $r_0=r_M=1$.
However, the left-hand side might contain several tensors, such as $\bm{\tilde\kappa}_{\ell}(\alphab)$ \eqref{eq:k_kle_pce_coeff}.
Then we can associate $s_0=\ell$ or $s_M=\ell$, and approximate all tensors for different $\ell$ in the same \emph{shared} TT representation.

High-dimensional matrices (cf. $\K$ in \eqref{eq:K_gen}) can be also presented in the TT format,
$$
\K = \sum_{s_0,\ldots,s_{M-1}} \bm{K}^{(0)}_{s_0} \otimes \bm{K}^{(1)}_{s_0,s_1} \otimes \cdots \otimes \bm{K}^{(M)}_{s_{M-1}}.
$$
The matrix by vector product $\K \uh$ is then recast to TT blocks of $\K$ and $\uh$.
Similarly, using the multilinearity of the TT format, we can cast linear combinations of initial tensors to concatenations of TT blocks.

The principal favor of the TT format comparing to the Canonical Polyadic (CP) decomposition (which is also popular, see e.g. \cite{wahnert-stochgalerkin-2014, litvinenko-spde-2013}) is a stable quasi-optimal rank reduction procedure \cite{osel-tt-2011}, based on singular value decompositions.
The complexity scales as $\mathcal{O}(Mpr^3)$, i.e. it is free from the curse of dimensionality,
while the full accuracy control takes place.
This procedure can be used to reduce unnecessarily large ranks after the matrix by vector product or the linear combination in the TT format, and to compress a tensor if it is fully given as a multidimensional array and fits into the memory.
However, in our situation this is not the case, and we need another approach to construct a TT format.

\subsection{Cross interpolation of matrices}
If $M=2$, the left-hand side of \eqref{eq:tt} can be seen as a matrix. For simplicity we consider this case first.
The basic assumption is that any entry of a matrix (or tensor) can be evaluated.
However, to construct a TT approximation, we do not want to compute all elements, but only few of them.

The principal ingredient for this is based on the efficiency of an incomplete Gaussian elimination in an approximation of a low-rank matrix, also known as the Adaptive Cross Approximation (ACA) \cite{bebe-2000,bebe-aca-2011}.
Given is a matrix $\U=[\U(i,j)]\in\mathbb{R}^{p \times q}$, we select some ``good'' columns and rows to approximate the whole matrix,
\begin{equation}
\U(i,j) \approx \bm{\tilde U}(i,j)= \U(i,\J) \cdot \widehat \U^{-1} \cdot \U(\I,j),
\label{eq:2dcross}
\end{equation}
where $\widehat \U = \U(\I, \J)$, and $\I \subset \{1,\ldots,p\}$, $\J \subset \{1,\ldots,q\}$ are sets of indices of cardinality $r$, so e.g. $\U(i,\J) \in \mathbb{R}^{1 \times r}$.
It is known that there exists a quasi-optimal set of interpolating indices $\I,\J$.
\begin{lemma}[Maximum volume (\emph{maxvol}) principle \cite{gt-maxvol-2001} ]
$\;$ If $\I$ and $\J$ are such that $\mathrm{det}~\U(\I, \J)$ is maximal among all $r \times r$ submatrices of $\U$, then
 $$
 \|\U-\bm{\tilde U}\|_C \le (r+1) \min_{\mathrm{rank}(\bm{V}) = r} \|\U-\bm{V}\|_2,
 $$
where $\| \cdot \|_C$ is the Chebyshev norm, $\|\bm{X}\|_C = \max_{i,j}|\bm{X}_{i,j}|$.
\label{lem:maxvol}
\end{lemma}

In practice, however, the computation of the true maximum volume submatrix is infeasible, since it is an NP-hard problem.
Instead one performs a heuristic iteration in an alternating fashion \cite{gostz-maxvol-2010}:
we start with some (e.g. random) low-rank factor $\U^{(1)} \in \mathbb{R}^{p \times r}$, determine indices $\I$ yielding a quasi-maxvol $r \times r$ submatrix in $\U^{(1)}$, and compute $\U^{(2)}$ as $r$ columns of $\U$ of the indices $\I$.
Vice versa, in the next step we find quasi-maxvol column indices in $\U^{(2)}$ and calculate corresponding $pr$ elements, collecting them into the newer $\U^{(1)}$, which hopefully approximates the true low-rank factor better than the initial guess.
This process continues until the convergence, which appears to be quite satisfactory in practice.

\subsection{TT-cross interpolation of tensors}

In higher dimensions we recurrently proceed in the same way, since the TT format constitutes a recurrent matrix low-rank factorization.
Let us merge the first $m$ and the last $M-m$ indices from $\alphab$.
The corresponding TT blocks will induce the following matrices:

\begin{definition}\label{def:interface}
Given a TT format \eqref{eq:tt} and an index $m$.
Define the \emph{left interface matrix} $\U^{<m} \in \mathbb{R}^{(p_1+1)\cdots(p_{m-1}+1) \times r_{m-1}}$ and the \emph{right interface} matrix $\U^{>m} \in \mathbb{R}^{r_m \times (p_{m+1}+1)\cdots(p_M+1)}$ as follows,
\begin{equation*}
 \begin{split}
  \U^{<m}_{s_{m-1}}(\alpha_1,\ldots,\alpha_{m-1}) & = \sum\limits_{s_1=1}^{r_1} \cdots \sum\limits_{s_{m-2}=1}^{r_{m-2}} \u^{(1)}_{s_1}(\alpha_1) \cdots \u^{(m-1)}_{s_{m-2},s_{m-1}}(\alpha_{m-1}), \\
  \U^{>m}_{s_{m}}(\alpha_{m+1},\ldots,\alpha_M) & = \sum\limits_{s_{m+1}=1}^{r_{m+1}} \cdots \sum\limits_{s_{M-1}=1}^{r_{M-1}} \u^{(m+1)}_{s_m,s_{m+1}}(\alpha_{m+1}) \cdots \u^{(M)}_{s_{M-1}}(\alpha_M).
 \end{split}
\end{equation*}
\end{definition}

Such matrices are convenient to relate high-dimensional TT expressions with their two-dimensional analogs.
For example, the TT format \eqref{eq:tt} can be written in the following form,
\begin{equation}
\uh(\alphab) \approx \U^{<m}(\alpha_1,\ldots,\alpha_{m-1}) \cdot \u^{(m)}(\alpha_m) \cdot \U^{>m}(\alpha_{m+1},\ldots,\alpha_M).
\label{eq:tt-l-b-r}
\end{equation}
The \emph{TT-cross} algorithm \cite{ot-ttcross-2010} assumes that the above expansion is valid on some $r_{m-1}(p_m+1)r_m$ indices, and can thus be seen as a system of equations on elements of $u^{(m)}$.
Let us be given $r_{m-1}$ \emph{left} indices $(\hat\alpha_1,\ldots,\hat\alpha_{m-1}) \in \I^{(m-1)}$ and $r_m$ \emph{right} indices $(\hat\alpha_{m+1},\ldots,\hat\alpha_{M}) \in \J^{(m+1)}$.
For the single index $\alpha_m$ we allow its full range $\{\alpha_m\} = \{0,1,\ldots,p_m\}$.
Requiring that \eqref{eq:tt-l-b-r} is valid on the combined indices
$$
\left(\hat\alpha_1,\ldots,\hat\alpha_{m-1},\alpha_m,\hat\alpha_{m+1},\ldots,\hat\alpha_{M}\right) \equiv \left(\I^{(m-1)}, \alpha_m, \J^{(m+1)}\right),
$$
we obtain the following computational rule for $\u^{(m)}$,
\begin{equation}
\u^{(m)}(\alpha_m) = \widehat \U_{<m}^{-1} \cdot \uh\left(\I^{(m-1)}, \alpha_m, \J^{(m+1)}\right) \cdot \widehat \U_{>m}^{-1},
\label{eq:ttcross}
\end{equation}
where $\widehat \U_{<m} = \U^{<m}(\I^{(m-1)})$ and $\widehat \U_{>m} = \U^{>m}(\J^{(m+1)})$ are the submatrices of $\U^{<m}$ and $\U^{>m}$ at indices $\I^{(m-1)}$ and $\J^{(m+1)}$, respectively.

Of course, these submatrices must be nonsingular.
In the previous subsection we saw a good strategy: if indices $\I^{(m-1)}$ and $\J^{(m+1)}$ are chosen in accordance with the maxvol principle for $\U^{<m}$ and $\U^{>m}$, the submatrices are not only nonsingular, but also 
provide a good approximation for \eqref{eq:tt-l-b-r}.
However, in practice $\U^{<m}$ and $\U^{>m}$ are too large to be treated directly.
Instead, we use the \emph{alternating} iteration over the dimensions and build nested index sets.

The alternating iteration means that we loop over $m=1,2,\ldots,M$ (the so-called \emph{forward} iteration) and $m=M,M-1,\ldots,1$ (\emph{backward} iteration).
Consider first the forward transition, $m-1 \rightarrow m$.
Given the set $\I^{(m-1)}$, the \emph{nested} indices $\I^{(m)}$ are defined as follows.
We concatenate $\I^{(m-1)}$ and $\alpha_m$, i.e. consider $r_{m-1} (p_m+1)$ indices $(\hat \alpha_{1},\ldots,\hat\alpha_{m-1},\alpha_m) \in \{\I^{(m-1)}, \alpha_m\}$, and select $r_m$ indices $\I^{(m)}$ only from $\{\I^{(m-1)}, \alpha_m\}$, not from all $\mathcal{O}(p^m)$ possibilities.
It can be seen that the next interface $\U^{<m+1}$, restricted to $\{\I^{(m-1)}, \alpha_m\}$, can be computed as a product of the previous submatrix $\widehat \U_{<m}$ and the current TT block $\u^{(m)}$.
To formalize this, we need the following definition.

\begin{definition}\label{def:folded}
Given a three-dimensional tensor $\u^{(m)} \in \mathbb{R}^{r_{m-1} \times (p_m+1) \times r_m}$, introduce the following reshapes, both pointing to the same data stored in $\u^{(m)}$:
\begin{itemize}
 \item \emph{left folding}: $\u^{|m\rangle}(s_{m-1},\alpha_m;~s_m) = \u^{(m)}_{s_{m-1},s_m}(\alpha_m), \quad \u^{|m\rangle} \in \mathbb{R}^{r_{m-1}(p_m+1) \times r_m},$  and
 \item \emph{right folding}: $\u^{\langle m |}(s_{m-1};~\alpha_m,s_m) = \u^{(m)}_{s_{m-1},s_m}(\alpha_m), \quad \u^{\langle m |} \in \mathbb{R}^{r_{m-1}\times (p_m+1) r_m}.$
\end{itemize}
\end{definition}

Then the restriction of $\U^{<m+1}$ writes as follows,
$$
\bm{V}^{\langle m|} = \widehat \U_{<m} \u^{\langle m |}, \quad \mbox{and} \quad \U^{<m+1}\left(\I^{(m-1)}, \alpha_m\right) = \bm{V}^{| m \rangle} \in \mathbb{R}^{r_{m-1}(p_m+1) \times r_m}.
$$
Thus, it is enough to apply the maximum volume algorithm \cite{gostz-maxvol-2010} to the matrix $\bm{V}^{| m \rangle}$, deriving \emph{local} maxvol indices $\hat i_m \subset \{1,\ldots, r_{m-1}(p_m+1)\}$, and obtain both $\I^{(m)}$ and $\widehat \U_{<m+1}$ by the restriction
$$
\I^{(m)} = \{\I^{(m-1)}, \alpha_m\}(\hat i_m), \quad \widehat \U_{<m+1} = \bm{V}^{| m \rangle}(\hat i_k) \in \mathbb{R}^{r_m \times r_m}.
$$

The backward transition $m+1 \rightarrow m$ for $\J^{(m)}$ and $\widehat \U_{>m}$ can be written analogously.
We show it directly in Alg.\ref{alg:cross} below.
In total, we need only $\mathcal{O}(n_{it} M p r^2)$ entries of $\uh$ to be evaluated, where $n_{it}$ is the number of alternating iterations, typically of the order of $10$.

\subsection{Rank-adaptive DMRG-cross algorithm}

A drawback of the TT-cross method is that the TT ranks are fixed; they have to be guessed a priori, which is also a problem of exponential complexity in $M$.
A remedy is to consider larger portions of data in each step.
The Density Matrix Renormalization Group (DMRG) algorithm was developed in the quantum physics community (\cite{white-dmrg-1993}, see also the review \cite{schollwock-2011} and
the references therein) to solve high-dimensional eigenvalue problems coming from the stationary spin Schroedinger equation.
It is written in a similar alternating fashion as the TT-cross procedure described above.
The crucial difference is that instead of one TT block as in \eqref{eq:ttcross} we calculate two neighboring TT blocks at once.

In the \emph{DMRG-cross} \cite{so-dmrgi-2011proc} interpolation algorithm, this is performed as follows.
Given $\mathbb{I}^{(m-1)}$, $\mathbb{J}^{(m+2)}$, we compute
$$
\u^{(m,m+1)}(\alpha_m,\alpha_{m+1}) =  \widehat \U_{<m}^{-1} \cdot \uh\left(\mathbb{I}^{(m-1)}, \alpha_m, \alpha_{m+1}, \mathbb{J}^{(m+2)}\right) \cdot \widehat \U_{>m+1}^{-1}.
$$
Then we need to separate indices $\alpha_m$ and $\alpha_{m+1}$ to recover the initial TT structure.
This can be done via the singular value decomposition.
The four-dimensional array $\u^{(m,m+1)}$ is reshaped to a matrix $\U^{(m,m+1)} \in \mathbb{R}^{r_{m-1}(p_m+1) \times (p_{m+1}+1) r_{m+1}}$, and the truncated SVD is computed,
$$
\U^{(m,m+1)} \approx \bm{V}\bSigma \bm{W}^\top, \quad \mbox{s.t.} \quad \left\|\U^{(m,m+1)} - \bm{V}\bSigma \bm{W}^\top\right\|_F\le \eps \left\|\U^{(m,m+1)}\right\|_F,
$$
where $\bm{V} \in \mathbb{R}^{r_{m-1} (p_m+1) \times \hat r_m}$, $\bSigma \in \mathbb{R}^{\hat r_m \times \hat r_m}$, $\bm{W} \in \mathbb{R}^{(p_{m+1}+1) r_{m+1} \times \hat r_m}$, and $\|\cdot\|_F$ is the Frobenius norm.
The new TT rank $\hat r_m$ is likely to differ from the old one $r_m$.
After that, we rewrite $\u^{(m)}$ and $\u^{(m+1)}$ with $\bm{V}$ and $\bm{W}$, respectively, replace $r_m=\hat r_m$, and continue the iteration.

To ensure that the perturbations introduced to $\u^{(m,m+1)}$ and the whole tensor $\uh$ coincide, we need to ensure that the interfaces $\U^{<m}$ and $\U^{>m+1}$ have orthonormal columns and rows, respectively.
Fortunately, this requires only small-sized QR decompositions of matrices $\u^{|m\rangle}$ and $\u^{\langle m |}$ \cite{osel-tt-2011,schollwock-2011} in the first iteration.
Later, the SVD will provide orthogonal factors in $\u^{(m)}$ and $\u^{(m+1)}$ automatically.

This algorithm allows a fast adaptation of TT ranks towards the requested accuracy threshold.
The price is, however, a larger degree of $p$ in the complexity, since we perform a full search in both $\alpha_m$ and $\alpha_{m+1}$ in each step.
The greedy-cross method \cite{sav-qott-2014} avoids this problem by maximizing the error over only $\mathcal{O}(rp)$ random entries among all $\mathcal{O}(r^2 p^2)$ elements in $\u^{(m,m+1)}$.
Here, instead of the neighboring block $\u^{(m+1)}$, it is the error that provides additional information and improves the approximation.
However, for the KLE-PCE coefficient \eqref{eq:k_kle_pce_coeff}, we have other candidates for such auxiliary data.
And we have reasons to consider them prior to the error.

\subsection{Block TT-Cross interpolation algorithm}
Note that each call of \eqref{eq:k_kle_pce_coeff} throws $L$ values, corresponding to different $\ell=1,\ldots,L$.
We may account for this in various ways.
Since $\ell$ has the meaning of the reduced spatial variable, we may feature it as an additional dimension.
But when we restrict the indices $\left\{\I^{(m-1)}, \alpha_m\right\}(\hat i_m) = \I^{(m)}$, we will remove some values of $\ell$ from consideration.
Therefore, a vast majority of information cannot be used: we evaluate $L$ values, but only a few of them will be
used to improve the approximation.
Another way is to run $L$ independent cross (e.g. DMRG-cross) algorithms to approximate each $\bm{\tilde \kappa}_{\ell}(\alphab)$ in its own TT format.
Yet, this is also not very desirable.
First, the TT ranks for the whole $\bm{\kappa}$ are usually comparable to the ranks of individual TT formats.
Therefore, $L$ cross algorithms consume almost $L$ times more time than the single run.
Secondly, we will have to add $L$ TT formats to each other by summing their ranks, which is to be followed by the TT approximation procedure.
The asymptotic complexity thus scales as $\mathcal{O}(L^3)$, which was found to be too expensive in practice.

A better approach is to store all $\bm{\tilde \kappa}_{\ell}(\alphab)$ in the same TT representation, employing the idea of the \emph{block} TT format \cite{dkos-eigb-2014}.
The resulting method has a threefold advantage: all data in each call of \eqref{eq:k_kle_pce_coeff} are assimilated, the algorithm adjusts the TT ranks automatically according to the given accuracy, and the output is returned as a single optimally-compressed TT format, convenient for further processing.

\begin{algorithm}[t]
 \caption{Block cross approximation of tensors in the TT format} \label{alg:cross}
 \begin{algorithmic}[1]
  \REQUIRE A function to evaluate $\uh_\ell(\alpha_1,\ldots,\alpha_M)$, initial TT guess $\u^{(1)}(\alpha_1)\cdots \u^{(M)}(\alpha_M)$, relative accuracy threshold $\eps$.
  \ENSURE  Improved TT approximation $\u^{(1)}_{\ell}(\alpha_1) \u^{(2)}(\alpha_2) \cdots \u^{(M)}(\alpha_M)$.
  \STATE Initialize $\mathbb{I}^{(0)}=[]$, \quad $\widehat \U_{<1}=1$, \quad $\mathbb{J}^{(M+1)} = []$, \quad $\widehat \U_{>M} = 1$.
\FOR{$\mathrm{iteration}=1,2,\ldots,n_{it}$ or until convergence}
  \FOR[Forward sweep]{$m=1,2,\ldots,M-1$}
    \IF[All indices are available, assimilate the information]{$\mathrm{iteration}>1$}
      \STATE Evaluate the tensors at cross indices and compute the common block by \eqref{eq:ttb_cross}.
      \STATE Compute the truncated SVD \eqref{eq:btt_svd} with accuracy $\eps$.
    \ELSE[Warmup sweep: the indices are yet to be built]
      \STATE Find QR decomposition $\u^{|m\rangle} = \bm{q}^{|m\rangle} \bm{R}$, $\left(\bm{q}^{|m\rangle}\right)^* \bm{q}^{|m\rangle } =\bm{I}$.
      \STATE Replace $\u^{\langle m+1 |} = \bm{R} \u^{\langle m+1 |}$, \quad  $\u^{|m\rangle}=\bm{q}^{|m\rangle}$.
    \ENDIF
    \STATE Compute the pre-restricted interface $\bm{V}^{\langle m |} = \widehat \U_{<m} \u^{\langle m |}$.
    \STATE Find \emph{local} maxvol indices $\hat i_m = \mathtt{maxvol}\left(\bm{V}^{|m\rangle}\right)$.
    \STATE New indices $\I^{(m)} = \left\{\I^{(m-1)}, \alpha_m\right\}(\hat i_m)$, interface $\widehat \U_{<m+1} = \bm{V}^{|m\rangle}(\hat i_m) \in \mathbb{R}^{r_{m} \times r_{m}}$.
  \ENDFOR
  \FOR[Backward sweep]{$m=M,M-1,\ldots,2$}
    \STATE Evaluate the tensors at cross indices and compute the common block by \eqref{eq:ttb_cross}.
    \STATE Compute the truncated SVD \eqref{eq:btt_svd_back} with accuracy $\eps$.
    \STATE Compute the pre-restricted interface $\bm{V}^{|m\rangle} = \u^{| m \rangle}  \widehat \U_{>m}$
    \STATE Find \emph{local} maxvol indices $\hat j_m = \mathtt{maxvol}\left(\left(\bm{V}^{\langle m |}\right)^\top\right)$.
    \STATE Restrict $\mathbb{J}^{(m)} = \left\{\alpha_m, \mathbb{J}^{(m+1)}\right\}(\hat j_m)$, \quad $\widehat \U_{>m-1} = \bm{V}^{\langle m |}(\hat j_m) \in \mathbb{R}^{r_{m-1} \times r_{m-1}}$.
  \ENDFOR
  \STATE Evaluate the first TT block $\u^{(1)}_{\ell}(\alpha_1) = \uh_\ell \left(\alpha_1, \mathbb{J}^{(2)} \right) \widehat \U_{>1}^{-1}$.
\ENDFOR
 \end{algorithmic}
\end{algorithm}

Assume that we have a procedure that, given an index $\alpha_1,\ldots,\alpha_M$, throws $L$ values $\uh_\ell(\alphab)$, $\ell=1,\ldots,L$.
When the tensor entries $\uh_{\ell}(\mathbb{I}^{(m-1)},\alpha_m,\mathbb{J}^{(m+1)})$ are evaluated, we modify \eqref{eq:ttcross} as follows:
\begin{equation}
\bm{y}^{(m)}(\alpha_m,\ell) =  \widehat \U_{<m}^{-1} \cdot \uh_{\ell}\left(\mathbb{I}^{(m-1)},\alpha_m,\mathbb{J}^{(m+1)}\right) \cdot \widehat \U_{>m}^{-1}.
\label{eq:ttb_cross}
\end{equation}
Now $\bm{y}^{(m)}$ is a four-dimensional tensor, in the same way as in the DMRG-cross.
We need to find a basis in $\alphab$ that is best suitable for all $\uh_{\ell}$.
Hence, we reshape $\bm{y}^{(m)}$ to the matrix $\bm{Y}^{(m)} \in \mathbb{R}^{r_{m-1}(p_m+1) \times L r_m}$ and compute the truncated singular value decomposition
\begin{equation}
\bm{Y}^{(m)} \approx \u^{|m\rangle} \bSigma \bm{W}^\top, \qquad \u^{|m\rangle} \in \mathbb{R}^{r_{m-1}(p_m+1) \times \hat r_m}, \quad \bSigma \in \mathbb{R}^{\hat r_m \times \hat r_m}, \quad \bm{W} \in \mathbb{R}^{L r_m \times \hat r_m}.
\label{eq:btt_svd}
\end{equation}
Again, the new rank $\hat r_m$ satisfies the Frobenius-norm error criterion, and replaces $r_m$ for the next iteration.
In the backward iteration, we reshape $\bm{y}^{(m)}$ to $\bm{Y}^{(m)} \in \mathbb{R}^{L r_{m-1} \times (p_m+1) r_m}$ and take the right singular vectors to the new TT block,
\begin{equation}
\bm{Y}^{(m)} \approx \bm{W} \bSigma \u^{\langle m |}, \qquad \bm{W} \in \mathbb{R}^{L r_{m-1} \times \hat r_m}, \quad \bSigma \in \mathbb{R}^{\hat r_m \times \hat r_m}, \quad \u^{\langle m |} \in \mathbb{R}^{\hat r_m \times (p_m+1) r_m}.
\label{eq:btt_svd_back}
\end{equation}
The whole procedure is summarized in the \emph{Block TT-Cross} Algorithm \ref{alg:cross}.

\section{TT-structured calculations with PCE}

\subsection{Computation of the PCE in the TT format via the cross interpolation}\label{sec:tt-kle-pce}
Equipped with Alg. \ref{alg:cross}, we may apply it to the PCE approximation, passing Formula \eqref{eq:k_kle_pce_coeff} as a function $u_\ell(\alphab)$ that evaluates tensor values on demand.
The initial guess may be even a rank-1 TT tensor with all blocks populated by random numbers, since the cross iterations will adapt both the representation and TT ranks.

Considering the sizes of the involved matrices, the complexity estimate can be written straightforwardly.
\begin{statement}
The cost to compute $\bm{\tilde\kappa}_{\ell}(\alphab)$ via the block TT-Cross Algorithm \ref{alg:cross} is
$$
\mathcal{O}(r^2 p (MN+NL) + r^3 pL + r^3 p L \cdot \min\{p,L\}).
$$
\end{statement}
 
The first two terms come from Formula \eqref{eq:k_kle_pce_coeff}, and the last one is the complexity of SVDs \eqref{eq:btt_svd} and \eqref{eq:btt_svd_back}.
\begin{remark}
$\;$It is unclear in general which term will dominate.
For large $N$, we are typically expecting that it is the evaluation \eqref{eq:ttb_cross}.
However, if $N$ is moderate (below $1000$), but the rank is large ($\sim 100$), the SVD consumes most of the time.
For the whole algorithm, assuming also $L \sim M$, we can thus expect the $\mathcal{O}(n_{it} M^2 N p r^3)$ complexity,
which is lower than $\mathcal{O}(M p L^3)$, which we could receive if we run independent cross algorithms for each $\ell$.
\end{remark}

As soon as the reduced PCE coefficients $\bm{\tilde\kappa}_{\ell}(\alphab)$ are computed, the initial expansion \eqref{eq:k_kle_pce} comes easily.
Indeed, after the backward cross iteration, the $\ell$ index belongs to the first TT block, and we may let it play the role of the ``zeroth'' TT rank index,
\begin{equation}
\bm{\tilde\kappa}_{\ell}(\alphab) = \sum_{s_1,\ldots,s_{M-1}} \bm{\kappa}^{(1)}_{\ell,s_1}(\alpha_1) \bm{\kappa}^{(2)}_{s_1,s_2}(\alpha_2) \cdots \bm{\kappa}^{(M)}_{s_{M-1}}(\alpha_M).
\end{equation}
For $\ell=0$ we extend this formula such that $\bm{\tilde\kappa}_0(\alphab)$ is the first identity vector $\bm{e}_0$, cf. \eqref{eq:f}.
Now we collect the spatial components from \eqref{eq:k_kle} into the ``zeroth'' TT block,
\begin{equation}
\kappa^{(0)}(x) = \begin{bmatrix} \kappa^{(0)}_\ell(x) \end{bmatrix}_{\ell=0}^{L} = \begin{bmatrix}\bar\kappa(x) & v_1(x) & \cdots & v_L(x) \end{bmatrix},
\end{equation}
then the PCE \eqref{eq:pce_k} writes as the following TT format,
\begin{equation}
\kappa(x,\alphab) = \sum_{\ell,s_1,\ldots,s_{M-1}} \kappa^{(0)}_{\ell}(x) \bm{\kappa}^{(1)}_{\ell,s_1}(\alpha_1) \cdots \bm{\kappa}^{(M)}_{s_{M-1}}(\alpha_M).
\label{eq:k_pce_tt}
\end{equation}

\subsection{Stiffness Galerkin operator in the TT format}\label{sec:stiff_tt}
With the full set $\Js_{M,\pb}$, we may benefit from the rank-1 separability of $\BDelta$, since each index $\alpha_m,\beta_m,\nu_m$ varies independently on the others.

\begin{lemma}
Given the PCE TT format \eqref{eq:k_pce_tt} for the coefficient $\kappa$ with the TT ranks $r(\kappa)$, the Galerkin operator \eqref{eq:K_gen} 
can be constructed in the TT format with the same ranks.
\end{lemma}

\begin{proof}
Given the PCE \eqref{eq:k_pce_tt} in the TT format, we split the whole sum over $\bm{\nu}$ in \eqref{eq:K_p} to the individual variables,
\begin{equation}
\begin{split}
\sum\limits_{\nu \in \Js_{M,\pb}} \BDelta_{\alphab,\betab,\bm{\nu}} \bm{\tilde \kappa}_{\ell}(\bm{\nu}) & = \sum_{s_1,\ldots,s_{M-1}} \bm{K}^{(1)}_{\ell,s_1}(\alpha_1,\beta_1) \bm{K}^{(2)}_{s_1,s_2}(\alpha_2,\beta_2) \cdots \bm{K}^{(M)}_{s_{M-1}}(\alpha_M,\beta_M), \\
\bm{K}^{(m)}(\alpha_m,\beta_m) & = \sum_{\nu_m=0}^{p_m} \Delta_{\alpha_m,\beta_m,\nu_m} \bm{\kappa}^{(m)}(\nu_m), \quad m=1,\ldots,M.
\end{split}
\label{eq:deltakappa_rank1}
\end{equation}
A similar reduction of a large summation to one-dimensional operations arises also in quantum chemistry \cite{vekh-lattice-2014}.
Agglomerate $\bm{K}^{(x)}_{\ell}(i,j)$ from \eqref{eq:K_l} to the ``zeroth'' TT block for the operator,
then the TT representation for the operator writes with the same TT ranks as in $\bm{\tilde \kappa}$,
\begin{equation}
\K = \sum_{\ell,s_1,\ldots,s_{M-1}} \bm{K}^{(x)}_{\ell} \otimes \bm{K}^{(1)}_{\ell,s_1} \otimes \cdots \otimes \bm{K}^{(M)}_{s_{M-1}} \in \mathbb{R}^{(N \cdot \#\Js_{M,\pb}) \times (N \cdot \#\Js_{M,\pb})}.
\label{eq:K_tt}
\end{equation}
\end{proof}

One interesting property of the Hermite triples is that $\Delta_{\alpha,\beta,\nu}=0$ if e.g. $\nu>\alpha+\beta$.
That is, if we set the same $p$ for $\alpha$, $\beta$ and $\nu$, in the assembly of \eqref{eq:K_gen} we may miss some components, corresponding to $\alpha>p/2$, $\beta>p/2$.
To compute the Galerkin operator exactly, it is reasonable to vary $\nu_m$ in the range $\{0,\ldots,2p\}$, and hence assemble $\tilde \kappa$ in the set $\Js_{M,2\pb}$.
While in the sparse set it would inflate the storage of $\BDelta$ and $\K$ significantly,
in the TT format it is feasible: the TT ranks do not depend on $p$, and the storage grows only linearly with $p$.


\subsection{Computation of the solution function}
Having solved the system \eqref{eq:linsystem}, we obtain the PCE coefficients of the solution in the TT format,
\begin{equation}
u(x,\alphab) = \sum_{s_0,\ldots,s_{M-1}} u^{(0)}_{s_0}(x) \u^{(1)}_{s_0,s_1}(\alpha_1) \cdots \u^{(M)}_{s_{M-1}}(\alpha_M).
\label{eq:tt_u}
\end{equation}
Some statistics are computable directly from $u(x,\alphab)$, but generally we need a function in the initial random variables, $u(x,\thetab)$.
Since $\Js_{M,\pb}$ is a tensor product set, $u(x,\alphab)$ can be turned into $u(x,\thetab)$ without changing the TT ranks, similarly to the construction of the 
Galerkin matrix in the previous subsection:
\begin{equation}
u(x,\thetab) = \sum_{s_0,\ldots,s_{M-1}} u^{(0)}_{s_0}(x) \left(\sum_{\alpha_1=0}^{p} h_{\alpha_1}(\theta_1) \u^{(1)}_{s_0,s_1}(\alpha_1)\right) \cdots \left(\sum_{\alpha_M=0}^{p} h_{\alpha_M}(\theta_M) \u^{(M)}_{s_{M-1}}(\alpha_M)\right).
\label{eq:alpha_to_theta}
\end{equation}

\subsection{Computation of statistics}\label{subsec:levelSets}
In this section we discuss how to calculate some statistical outputs from the solution in the TT format, such as the mean, the (co)variance and the probability of an event.

The \emph{mean} value of $u$, in the same way as in $\kappa$, can be derived as the PCE coefficient at $\alphab=(0,\ldots,0)$, $\bar{u}(x) = u(x,0,\ldots,0)$.
It requires no additional calculations.

The \emph{covariance} is more complicated and requires both multiplication (squaring) and summation over $\alphab$.
By definition, the covariance reads
\begin{equation*}
\begin{split}
\cov_u(x,y) & = \int\limits_{\mathbb{R}^M} \left(u(x,\thetab) - \bar u(x)\right) \left(u(y,\thetab) - \bar u(y)\right) \rho(\thetab) d\thetab \\
 & = \sum\limits_{\substack{\alphab,\betab \neq (0,\ldots,0), \\ \alphab, \betab \in \mathcal{J}_{M,\pb}}} u(x,\alphab) u(y,\betab) \int\limits_{\mathbb{R}^M} H_{\alphab}(\thetab) H_{\betab}(\thetab) \rho(\thetab) d\thetab.
\end{split}
\end{equation*}
Knowing that $\int H_{\alphab}(\thetab) H_{\betab}(\thetab) \rho(\thetab) d\thetab = \alphab! \delta_{\alphab,\betab}$,
we take $\u_{s_0}(\alphab) = \u^{(1)}_{s_0}(\alpha_1) \cdots \u^{(M)}(\alpha_M)$ from \eqref{eq:tt_u} and multiply it with the Hermite mass matrix (TT ranks do not change),
\begin{equation}
\bm{w}_{s_0}(\alphab) :=  \u_{s_0}(\alphab) \sqrt{\alphab!} =  \sum_{s_1,\ldots,s_{M-1}} \left(\u^{(1)}_{s_0,s_1}(\alpha_1) \sqrt{\alpha_1!} \right) \cdots  \left(\u^{(M)}_{s_{M-1}}(\alpha_M) \sqrt{\alpha_M!} \right),
\label{eq:param_chunk_masscorr}
\end{equation}
and then take the scalar product $\bm{C} = \left[\bm{C}_{s_0,s_0'}\right]$, where $\bm{C}_{s_0,s_0'} = \dotprod{\bm{w}_{s_0}}{\bm{w}_{s_0'}}$ with $\bm{w}$ defined in \eqref{eq:param_chunk_masscorr}.
Given the TT rank bound $r$ for $u(x,\alphab)$, we deduce the $\mathcal{O}(Mpr^3)$ complexity of this step.
After that, the covariance is given by the product of $\bm{C}$ with the spatial TT blocks,
\begin{equation}
\cov_u(x,y) = \sum\limits_{s_0,s_0'=0}^{r_0} u^{(0)}_{s_0}(x) \bm{C}_{s_0,s_0'} u^{(0)}_{s_0'}(y),
\label{eq:cov_tt}
\end{equation}
where $u^{(0)}_{s_0}$ is the ``zeroth'' (spatial) TT block of the decomposition \eqref{eq:tt_u}.
Given $N$ degrees of freedom for $x$, the complexity of this step is $\mathcal{O}(N^2 r_0^2)$.
Note that a low-rank tensor approximation of a large covariance matrix is very important, for example, in Kriging  \cite{LitvNowak13}.
The \emph{variance} is nothing else than the diagonal of the covariance, $\var_u(x) = \cov_u(x,x)$.

Other important outputs are the characteristic, level set functions, and the probability of a particular event \cite{litvinenko-spde-2013}.
\begin{definition}\label{defn:LevelSetFrequency}
Let $\mathbb{S} \subset \mathbb{R}$ be a subset of real numbers.
\begin{itemize}
\item The \emph{characteristic} function of $u$ at $\mathbb{S}$ is defined pointwise for all $\thetab \in \mathbb{R}^{M}$ as follows,
\begin{equation}\label{equ:chi}
    \chi_{\mathbb{S}}(x,\thetab):=\left\{
                         \begin{array}{ll}
                           1, & u(x,\thetab) \in \mathbb{S}, \\
                           0, & u(x,\thetab) \nin \mathbb{S}.
                         \end{array}
                       \right.
\end{equation}
\item The \emph{level set} function reads $\mathcal{L}_{\mathbb{S}}(x,\thetab):= u(x,\thetab) \chi_{\mathbb{S}}(x,\thetab)$.
\item The \emph{probability} of $\mathbb{S}$ reads $\mathbb{P}_{\mathbb{S}}(x) = \int_{\mathbb{R}^{M}} \chi_{\mathbb{S}}(x,\thetab) \rho(\thetab) d\thetab$.
\end{itemize}
\end{definition}

The characteristic function can be computed using either the cross Algorithm \ref{alg:cross} (see also \cite{Ballan_Grasedyck10, Ballan_Grasedyck14}), which takes Formula \eqref{equ:chi} as the function that evaluates a high-dimensional array $\chi$ at the indices in $x,\thetab$, or the Newton method for the sign function \cite{litvinenko-spde-2013}.
In both cases we may face a rapid growth of TT ranks during the cross or Newton iterations: the characteristic function is likely to have a discontinuity that 
is not aligned to the coordinate axes, and some of 
the singular values in the TT decomposition will decay very slowly.
We face the same problem with increasing ranks during computing the level set functions and exceedance probabilities.

However, the probability is easier to compute, especially if it is relatively small.
Using the same cross algorithm, we can compute directly the product $\hat\chi_{\mathbb{S}}(x,\thetab) = \chi_{\mathbb{S}}(x,\thetab) \rho(\thetab)$.
The probability (at a fixed $x$) is then computed as a scalar product with the all-ones vector in the TT format, $\mathbb{P}_{\mathbb{S}}(x) = \dotprod{\hat\chi}{\mathtt{1}}$.
But if $\mathbb{P}$ is small, it means that most of the entries in $\hat\chi$ are small, and do not inflate TT ranks, which might be the case for $\chi$.
Typically, the computation of small probabilities is used to predict the failure risk of a technical system.
The event set has the form $\mathbb{S} = \left\{z\in\mathbb{R}:~z>\tau\right\}$, and the probability is called the \emph{exceedance} probability.
This will be studied in numerical experiments.
%
%
%
\section{Numerical Experiments}
\label{sec:numerics}
We verify the approach on the elliptic stochastic equation \eqref{eq:elliptic} in a two-dimensional $L$-shape domain, $x=(x_1,x_2) \in D = [-1,1]^2 \backslash [0,1]^2$.
We pose zero Dirichlet boundary conditions and use the deterministic right-hand side $f = f(x) = 1$.
The stochasticity is introduced in the diffusion coefficient $\kappa(x,\omega)$; we investigate log-normal and beta distributions for $\kappa$.

To generate the spatial mesh, we use the standard PDE Toolbox in MATLAB.
We consider from $1$ to $3$ refinement levels of the spatial grid, denoted by $R$.
The first refinement $R=1$ yields $557$ degrees of freedom, $R=2$ gives $2145$ points, and $R=3$ corresponds to $8417$ points.
Since we have to store the stiffness matrices in a dense form, we cannot refine the grid further.

The KLE of both $\kappa$ and $\gamma$ is truncated to the same number of terms $L=M$.

All utilities related to the Hermite PCE were taken from the \emph{sglib} \cite{sglib}, including discretization and solution routines in the sparse polynomial set $\Js_{M,p}^{sp}$.
However, to work with the TT format (for full $\Js_{M,\pb}$), we employ the \emph{TT-Toolbox} \cite{tt-toolbox}.
The same polynomial orders are chosen in all variables, $\pb=(p,\ldots,p)$.
We use the modules of \emph{sglib} for low-dimensional stages and replace the parts corresponding to high-dimensional calculations by the TT algorithms.
The block TT-cross Alg. \ref{alg:cross} is implemented in the procedure \texttt{amen\_cross.m} from the TT-Toolbox, and the linear system \eqref{eq:linsystem} was solved via the Alternating Minimal Energy (AMEn) method \cite{ds-amen-2014}, the procedure \texttt{amen\_solve.m} from the companion package tAMEn \cite{tamen}.
Computations were conducted in MATLAB R2012a on a single core of the Intel Xeon X5650 CPU at 2.67GHz, provided by the Max- Planck-Institute, Magdeburg.

The accuracy of the coefficient and the solution was estimated using the Monte Carlo method with $N_{mc}$ simulations.
We approximate the average $L_2$-error as follows,
\begin{equation*}
 E_{\kappa} = \frac{1}{N_{mc}} \sum\limits_{z=1}^{N_{mc}} \dfrac{\sqrt{\sum_{i=1}^{N} \left(\kappa(x_i,\thetab_z) - \kappa_*(x_i,\thetab_z)\right)^2}}{\sqrt{\sum_{i=1}^{N} \kappa_*^2(x_i,\thetab_z)}} \approx \int\limits_{\mathbb{R}^M} \dfrac{\|\kappa(x,\thetab) - \kappa_*(x,\thetab)\|_{L_2(D)}}{\|\kappa_*(x,\thetab)\|_{L_2(D)}} \rho(\thetab) d\thetab,
\end{equation*}
where $\{\thetab_z\}_{z=1}^{N_{mc}}$ are normally distributed random samples, $\{x_i\}_{i=1}^{N}$ are the spatial grid points,
and $\kappa_*(x_i,\thetab_z) = \phi\left(\gamma(x_i,\thetab_z)\right)$ is the reference coefficient computed without using the PCE for $\phi$.
The same definition is used for the solution $u$, with $u_*(x,\thetab_z)$ being the solution of the deterministic PDE with the coefficient $\kappa_*(x,\thetab_z)$.

Besides that we compare the statistics obtained with our approaches and the Monte Carlo method.
For the mean and variance we use the same discrete approximation to the $L_2$-norm,
$$
E_{\bar u} = \dfrac{\|\bar u - \bar u_*\|_{L_2(D)}}{\|\bar u_*\|_{L_2(D)}}, \qquad E_{\var_u} = \dfrac{\|\var_u - \var_{u*}\|_{L_2(D)}}{\|\var_{u*}\|_{L_2(D)}}.
$$

To examine the computation of probabilities, we compute the exceedance probability.
This task can be simplified by taking into account the maximum principle: the solution is convex w.r.t. $x$.
We compute the maximizer of the mean solution, $x_{\max}:~\bar u(x_{\max}) \ge \bar u(x)~\forall x \in D$.
Fix $x$ to $x_{\max}$ and consider only the stochastic part $\uh_{\max}(\thetab) = u(x_{\max},\thetab)$, and $\hat u = \bar u(x_{\max})$.
Now, taking some $\tau>1$, we compute
\begin{equation}
\mathbb{P} = \mathbb{P}\left(\uh_{\max}(\thetab)>\tau \hat u \right) = \int\limits_{\mathbb{R}^M} \chi_{\uh_{\max}(\thetab)>\tau \hat u} (\thetab) \rho(\thetab) d\thetab.
\label{eq:rare_prob}
\end{equation}
By $\mathbb{P}_*$ we will also denote the probability computed from the Monte Carlo method, and estimate the error as 
$E_{\mathbb{P}} = \left|\mathbb{P}-\mathbb{P}_*\right|/\mathbb{P}_*$.

\subsection{Log-normal distribution}
Let
$$
\kappa(x,\omega) = \exp\left(1+\sigma \frac{\gamma(x,\omega)}{2}\right)+10,
$$
where $\gamma \sim \mathcal{N}(0,1)$ is the standard normally distributed random field.
The covariance function is taken Gaussian, $\cov_{\kappa}(x,y) = \exp\left(-\frac{\|x-y\|^2}{2 l_c^2} \right)$, where $l_c$ is the (isotropic) correlation length.

The default parameters are the following: number of KLE terms $M=20$, polynomial order $p=3$, correlation length $l_c=1$, dispersion $\sigma=0.5$, refinement level of the spatial grid $R = 1$, and the tensor approximation accuracy $\eps=10^{-4}$.
Below we will vary each of these parameters, keeping the others fixed.
For the computation of the probability \eqref{eq:rare_prob} we use $\tau=1.2$.

\subsubsection{Verification of the block cross algorithm}
Formula \eqref{eq:k_kle_pce_coeff} can be evaluated for each KLE index $\ell$ independently, using existing cross approximation algorithms.
We compare with the so-called DMRG cross method \cite{so-dmrgi-2011proc}, which is conceptually the closest approach to our Algorithm \ref{alg:cross}.
In Table \ref{tab:crosses} we show the performance of the single run of Algorithm \ref{alg:cross} (which gives the coefficient for all $\ell$ simultaneously) and of $L$ DMRG crosses, followed by the summation of individual terms to the common representation.
We see that even if the TT ranks of the output are exactly the same, the times differ dramatically.
This is because the ranks of individual components and of the whole coefficient are comparatively the same, and the DMRG approach requires roughly $L$ times more operations.
The last column in Table \ref{tab:crosses} confirms that both approaches deliver the same data up to the approximation tolerance.

\begin{table}[t]
\centering
\caption{Performance of the block cross Alg. \ref{alg:cross} versus the DMRG cross \cite{so-dmrgi-2011proc}}
\label{tab:crosses}
\begin{tabular}{c|cc|cc|c}
$M$  & \multicolumn{2}{c|}{Block cross} & \multicolumn{2}{c|}{DMRG crosses} & $\underline{\|\kappa_{DMRG}-\kappa_{Block}\|}$ \\
     &  CPU time, sec.  & $r_{\kappa}$ &   CPU time, sec.  & $r_{\kappa}$ & $\|\kappa_{Block}\|$ \\ \hline
10   &     4.908   &    20  &  31.29   & 20   &  2.77e-5  \\
15   &     10.36   &    27  &  114.9   & 27   &  2.24e-5  \\
20   &     19.11   &    32  &  286.2   & 33   &  1.83e-4  \\
30   &     49.89   &    39  &  1372.0  & 50   &  2.52e-4  \\
\end{tabular}
\end{table}

\subsubsection{Experiment with the polynomial order $p$}
First, we provide a detailed study of the computational time of each of the stages in the TT and Sparse methods: construction of the coefficient ($T_{\kappa}$), construction of the operator ($T_{op}$), and the solution of the system ($T_{u}$).
The results are shown in Table \ref{tab:log-normal-ttimes-p}, and times are measured in seconds.
The complexity of the cross algorithm, employed for the computation of $\kappa$ in the TT format, grows linearly with $p$, since the TT ranks are stable w.r.t. $p$ (see Table \ref{tab:log-normal-p}).
However, it is much slower than the direct evaluation of the coefficients in the sparse set.
This is mainly due to the singular value decompositions, involving matrices of sizes hundreds.

Nevertheless, for the computation of the Galerkin matrix the situation is the opposite.
In the TT format, the computations via formula \eqref{eq:deltakappa_rank1} are very efficient, since they involve $M$ products of $p^2 \times 2p$ matrices.
In the sparse representation, we have to perform all $(\#\Js_{M,p}^{sp})^3$ operations, which is very time-consuming.
Since $\#\Js_{M,p}^{sp}$ grows exponentially with $p$, we had to skip the cases with large $p$.

The solution stage is more simple, since the mean value preconditioner is quite efficient, both for the standard CG method with the sparse set and the AMEn method for the TT format.
Again, the complexity of the TT solver grows linearly with $p$.
The sparse solver works reasonably fast as well, but it cannot be run before the matrix elements are computed\footnote{It is sometimes advocated to 
avoid a construction of the matrix and to compute its elements only when they are needed in the 
matrix-vector product. It saves memory, but the computational time will be comparatively the same, since it is proportional to the number of operations.}; hence it is also skipped for $p=4,5$.

\begin{table}[t]
\centering
 \caption{Detailed CPU times (sec.) versus $p$, log-normal distribution}
 \label{tab:log-normal-ttimes-p}
 \begin{tabular}{c|ccc|ccc}
     & \multicolumn{3}{c|}{TT (full index set $\Js_{M,\pb}$)} & \multicolumn{3}{c}{Sparse (index set $\Js_{M,p}^{sp}$)} \\
 $p$ & $T_{\kappa}$ & $T_{op}$ & $T_{u}$ & $T_{\kappa}$ & $T_{op}$ & $T_{u}$ \\ \hline
 1 &  9.6166  &    0.1875 &  1.7381    &     0.4525  &    0.2830  &   0.6485 \\
 2 & 14.6635  &    0.1945 &  2.9584    &     0.4954  &    3.2475  &   1.4046  \\
 3 & 19.1182  &    0.1944 &  3.4162    &     0.6887  &    1027.7  &   18.1263 \\
 4 & 24.4345  &    0.1953 &  4.2228    &     2.1597  &      ---   &     ---   \\
 5 & 30.9220  &    0.3155 &  5.3426    &     9.8382  &      ---   &     ---   \\
 \end{tabular}
\end{table}

In Table \ref{tab:log-normal-p} we present the total CPU times required in both methods to find the solution $u$, the time for computing $\hat \chi$, maximal TT ranks of the coefficient ($r_{\kappa}$),
the solution ($r_u$) and the weighted characteristic function ($r_{\hat \chi}$), as well as statistical errors in the coefficient ($E_{\kappa}$) and the solution ($E_{u}$).
The probability $\mathbb{P}$ is presented only for the TT calculation.
Since $\mathbb{P} \sim 10^{-4}$, $10000$ simulations may be not enough to compute $\mathbb{P}$ with the Monte Carlo method.
Below we present a dedicated test of the Monte Carlo approach.

\begin{table}[t]
\centering
 \caption{Performance versus $p$, log-normal distribution}
 \label{tab:log-normal-p}
 \begin{tabular}{c|ccc|ccc|cc|cc|c}
  $p$ & \multicolumn{3}{c|}{CPU time, sec.} & $r_{\kappa}$ & $r_u$ & $r_{\hat \chi}$ & \multicolumn{2}{c|}{$E_{\kappa}$} & \multicolumn{2}{c|}{$E_{u}$} & $\mathbb{P}$ \\ \hline
      & TT       & Sparse & $\hat\chi$       &              &       &            &  TT & Sparse & TT & Sparse & TT  \\ \hline
  1   & 11.54  & 1.38     & 0.23   & 32   &   42  &   1  & 3.75e-3  & 1.69e-1  & 9.58e-3  & 1.37e-1  & 0       \\
  2   & 17.81  & 5.14     & 0.32   & 32   &   49  &   1  & 1.35e-4  & 1.10e-1  & 4.94e-4  & 4.81e-2  & 0       \\
  3   & 22.72  & 1046     & 83.12  & 32   &   49  &   462  & 6.21e-5  & 2.00e-3  & 2.99e-4  & 5.29e-4  & 2.75e-4     \\
  4   & 28.85  &   ---    & 69.74  & 32   &   50  &   416  & 6.24e-5  &     ---  & 9.85e-5  &     ---  & 1.21e-4     \\
  5   & 36.58  &   ---    & 102.5  & 32   &   49  &   410  & 6.27e-5  &     ---  & 9.36e-5  &     ---  & 6.20e-4     \\
 \end{tabular}
\end{table}

\subsubsection{Experiment with the KLE dimension $M$}
The length of the truncated KLE is another crucial parameter of the stochastic PDE.
In Table \ref{tab:log-normal-M} we compare the TT and Sparse procedures.

\begin{table}[t]
\centering
 \caption{Performance versus $M$, log-normal distribution}
 \label{tab:log-normal-M}
 \begin{tabular}{c|ccc|ccc|cc|cc|c}
  $M$ & \multicolumn{3}{c|}{CPU time, sec.} & $r_{\kappa}$ & $r_u$ & $r_{\hat \chi}$ & \multicolumn{2}{c|}{$E_{\kappa}$} & \multicolumn{2}{c|}{$E_{u}$} & $\mathbb{P}$ \\ \hline
       & TT     & Sparse & $\hat\chi$       &              &       &            &  TT & Sparse & TT & Sparse & TT  \\ \hline
  10   & 6.401  & 6.143  & 1.297   & 20 & 39     & 70  & 2.00e-4 &  1.71e-1  & 3.26e-4 &  1.45e-1 & 2.86e-4       \\
  15   & 12.15  & 92.38  & 22.99     & 27 & 42   & 381 & 7.56e-5 &  1.97e-3  & 3.09e-4 &  5.41e-4 & 2.99e-4       \\
  20   & 21.82  & 1005   & 67.34 & 32 & 50       & 422 & 6.25e-5 &  1.99e-3  & 2.96e-4 &  5.33e-4 & 2.96e-4   \\
  30   & 52.92  & 48961  & 136.5     & 39 & 50   & 452 & 6.13e-5 &  9.26e-2  & 3.06e-4 &  5.51e-2 & 2.78e-4
 \end{tabular}
\end{table}

We see that the accuracy of the TT approach is stable w.r.t. $M$, and the complexity grows mildly.
Note, however, that the correlation length $l_c=1$ is quite large and yields a fast decay of the KLE, such that $M=20$ is actually enough for the accuracy $10^{-4}$.
The TT approach demonstrates stability w.r.t. the overapproximation at $M=30$.
This is not the case for the sparse approach: at high $M$ the accuracy is lost.
This is because $p=3$ is not enough to transform the covariance \eqref{eq:cov_transform} accurately.
Since eigenvalues of the covariance decay rapidly, higher eigenpairs become strongly perturbed (the eigenvalues can even become negative), and a large error propagates to the PCE.
In the full set, the maximal polynomial order is equal to $pM$, and the error of the covariance transform is negligible.

\subsubsection{Experiment with the correlation length $l_c$}
The Gaussian covariance function yields an exponential decay of the KLE coefficients \cite{SCHWAB2006,Bieri2009},
but the actual rate is highly dependent on the correlation length \cite{Hcovariance, litvinenko2009sparse}.
In this experiment, we study the range of lengths from $0.1$ to $1.5$.
In order to have a sufficient accuracy for all values of $l_c$, we fix the dimension $M=30$.
The results are presented in the same layout as before in Table \ref{tab:log-normal-l}.

\begin{table}[t]
\centering
 \caption{Performance versus $l_c$, log-normal distribution}
 \label{tab:log-normal-l}
 \begin{tabular}{c|ccc|ccc|cc|cc|c}
  $l_c$ & \multicolumn{3}{c|}{CPU time, sec.} & $r_{\kappa}$ & $r_u$ & $r_{\hat \chi}$ & \multicolumn{2}{c|}{$E_{\kappa}$} & \multicolumn{2}{c|}{$E_{u}$} & $\mathbb{P}$ \\ \hline
       & TT      & Sparse & $\hat\chi$       &              &       &            &  TT & Sparse & TT & Sparse & TT  \\ \hline
  0.1   & 216  & 55826  & 0.91  & 70 & 50 &   1 & 1.98e-2 &   1.98e-2 & 1.84e-2 &  1.84e-2   & 0   \\
  0.3   & 317  & 52361  & 41.8  & 87 & 74 & 297 & 3.08e-3 &   3.51e-3 & 2.64e-3 &  2.62e-3   & 8.59e-31   \\
  0.5   & 195  & 51678  & 58.1  & 67 & 74 & 375 & 1.49e-4 &   2.00e-3 & 2.58e-4 &  3.10e-4   & 6.50e-33   \\
  1.0   & 57.3  & 55178  & 97.3  & 39 & 50 & 417 & 6.12e-5 &   9.37e-2 & 3.18e-4 &  5.59e-2   & 2.95e-04   \\
  1.5   & 32.4  & 49790  & 121  & 31 & 34 & 424 & 3.24e-5 &   2.05e-1 & 4.99e-4 &  1.73e-1   & 7.50e-04   \\
 \end{tabular}
\end{table}

In the TT approach, we see a clear decrease of the computational complexity and the error with growing covariance length.
This is because the SVD approximation in the TT format automatically reduces the storage w.r.t. the latter (less important) variables, if the KLE decay is fast enough.
The TT errors reflect the amount of information discarded in the truncated KLE tail, which is large for small $l_c$ and small otherwise.

The errors in the sparse approach behave in the same way until $l_c=0.5$, but for $l_c=1$ and $1.5$ the dimension $M=30$ is too large, and the instability w.r.t. the overapproximation takes place.

With fixed $M$, the exceedance probability is very sensitive to the correlation length.
Truncating the KLE, we reduce the total variance of the random field. For a (quasi)-Gaussian distribution, a small perturbation of the variance has
a small effect on the integral over the peak region,
but may have a very large relative effect on the tail region, which corresponds to the small exceedance probability.
\subsubsection{Experiment with the dispersion $\sigma$}
The variance of the normally distributed field $\gamma(x,\omega)$ is equal to $\sigma^2$.
Since it enters $\kappa$ under the exponential, it influences the variance of $\kappa$ significantly.
In Table \ref{tab:log-normal-sigma} we vary $\sigma$ from $0.2$ to $1$ and track the performance of the methods.
As expected, the computational complexity grows with $\sigma$, as does the contrast in the coefficient.
However, we were able to perform all computations for each value of $\sigma$.

\begin{table}[t]
\centering
 \caption{Performance versus $\sigma$, log-normal distribution}
 \label{tab:log-normal-sigma}
 \begin{tabular}{c|ccc|ccc|cc|cc|c}
  $\sigma$ & \multicolumn{3}{c|}{CPU time, sec.} & $r_{\kappa}$ & $r_u$ & $r_{\hat \chi}$ & \multicolumn{2}{c|}{$E_{\kappa}$} & \multicolumn{2}{c|}{$E_{u}$} & $\mathbb{P}$ \\ \hline
       & TT      & Sparse & $\hat\chi$       &              &       &            &  TT & Sparse & TT & Sparse & TT  \\ \hline
  0.2   &    15.93  & 1008   & 0.348   & 21& 31 &   1 &    5.69e-5 &  4.76e-5 &  4.19e-5 &  1.30e-5  & 0   \\
  0.4   &    18.72  & 968.3  & 0.341   & 29& 42 &   1 &    6.88e-5 &  8.04e-4 &  1.40e-4 &  2.14e-4  & 0   \\
  0.5   &    21.23  & 970.1  & 79.96   & 32& 49 & 456 &    6.19e-5 &  2.02e-3 &  3.05e-4 &  5.45e-4  & 2.95e-4   \\
  0.6   &    24.08  & 961.5  & 24.72   & 34& 57 & 272 &    9.12e-5 &  4.42e-3 &  6.14e-4 &  1.16e-3  & 2.30e-3   \\
  0.8   &    31.69  & 969.1  & 67.93   & 39& 66 & 411 &    4.40e-4 &  8.33e-2 &  2.02e-3 &  2.90e-2  & 8.02e-2   \\
  1.0   &    50.67  & 1071   & 48.44   & 44& 82 & 363 &    1.73e-3 &  4.10e-1 &  4.96e-3 &  3.08e-1  & 9.17e-2
 \end{tabular}
\end{table}

\subsubsection{Experiment with the spatial grid refinement $R$}
Since the efforts of dealing with the full spatial matrix grow significantly with the grid refinement,
in this test we limit ourselves to $M=10$.
The principal observations from Table \ref{tab:log-normal-R_l} are that the TT rank and the accuracy\footnote{Note that the errors are estimated via the Monte Carlo method on the same grids, thus they show the accuracy of the PCE approximation, not the spatial discretization.} are stable w.r.t. the grid refinement.
Therefore, we may expect that the TT approach will also be efficient for finer grids, if we find an efficient way to deal with the spatial dimension.
A research on \emph{non-intrusive} stochastic Galerkin methods, addressing this issue, has begun recently \cite{matthies-to-be-intrusive-1-2014}, and we plan to adopt it in the TT framework in future.

\begin{table}[t]
\centering
 \caption{Performance versus $R$, log-normal distribution. The left column shows the number of spatial degrees of freedom ($\#$DoF) for $R=1,2,3$.}
 \label{tab:log-normal-R_l}
 \begin{tabular}{c|ccc|ccc|cc|cc|c}
  $\#$DoF & \multicolumn{3}{c|}{CPU time, sec} & $r_{\kappa}$ & $r_u$ & $r_{\hat \chi}$ & \multicolumn{2}{c|}{$E_{\kappa}$} & \multicolumn{2}{c|}{$E_{u}$} & $\mathbb{P}$ \\ \hline
      & TT     & Sparse & $\hat\chi$       &              &       &            &  TT & Sparse & TT & Sparse & TT  \\ \hline
  557   & 6.40  & 6.09  & 1.29    & 20 & 39 & 71   &    2.00e-4 &  1.71e-1  &    3.26e-4 &  1.45e-1  & 2.86e-4       \\
  2145   & 8.98  & 13.7  & 1.17    & 20 & 39 & 76   &    1.74e-4 &  1.89e-3  &    3.33e-4 &  5.69e-4  & 2.90e-4       \\
  8417   & 357  & 171  & 0.84    & 20 & 40 & 69   &    1.65e-4 &  1.88e-3  &    3.24e-4 &  5.64e-4  & 2.93e-4   \\
 \end{tabular}
\end{table}

\subsubsection{Comparison with the Monte Carlo Method}
For the Monte Carlo test, we prepare the TT solution with parameters $p=5$ and $M=30$.
The results are presented in Table \ref{tab:log-normal-mc}.
In the left part of the table we show the performance of the Monte Carlo method with different numbers of simulations: total CPU time ($T_{MC}$), 
errors in the mean and variance of $u$, and a small exceedance probability with its error.
The right part contains the results of the TT approach: the aggregated CPU time of construction of the coefficient, operator and solution ($T_{solve}$), the time to compute the weighted characteristic function $(T_{\hat\chi})$, 
TT ranks of all data and the probability calculated from $\hat\chi$.

\begin{table}[t]
\centering
 \caption{Verification of the Monte Carlo method, log-normal distribution}
 \label{tab:log-normal-mc}
 \begin{tabular}{c|l|cc|cc||c|c}
  $N_{mc}$ &  $T_{MC}$, sec.		& $E_{\bar u}$ & $E_{\var_u}$  & $\mathbb{P}_*$ & $E_{\mathbb{P}}$ & \multicolumn{2}{c}{TT results} \\  \hline
  $10^2$   &  0.6398                    & 9.23e-3   & 1.49e-1     & 0               & $\infty$ & $T_{solve}$ &  96.89 sec. \\
  $10^3$   &  6.1867                    & 1.69e-3   & 5.97e-2     & 0               & $\infty$ & $T_{\hat\chi}$            &  157.0 sec. \\
  $10^4$   &  6.1801$\cdot 10^1$        & 5.81e-4   & 7.12e-3     & 4.00e-4         & 5.53e-1  & $r_{\kappa}$              &  39     \\
  $10^5$   &  6.2319$\cdot 10^2$        & 2.91e-4   & 2.91e-3     & 4.10e-4         & 5.15e-1  & $r_{u}$                   &  50     \\
  $10^6$   &  6.3071$\cdot 10^3$        & 1.23e-4   & 9.76e-4     & 4.60e-4         & 3.51e-1  & $r_{\hat\chi}$  	      &  432    \\
           &                            &              &                &                   &              & $\mathbb{P}$             &  6.214e-4
 \end{tabular}
\end{table}

We see that the cost of the TT method is comparable with the cost of $40000$ Monte Carlo tests.
That many realizations already provide a good approximation of the mean, a bit less accurate for the variance, but it is by far not sufficient for a confident estimate of the exceedance probability.
Therefore, the tensor product methods can be recommended as a competitive alternative to classical techniques for computing exceedance probabilities.


\subsection{Beta distribution}
The Hermite expansion \eqref{eq:phi_tranform} of the $\exp$ function in the log-normal case yields the coefficients of the form $\phi_i = \frac{c}{i!}$.
Therefore, the PCE coefficient formula \eqref{eq:pce_k_coeffs} resolves to a direct product of univariate functions of $\alpha_1,\ldots,\alpha_M$,
and the tensor format of the PCE can be constructed explicitly \cite{wahnert-stochgalerkin-2014}.
To demonstrate the generality of the cross algorithm, we also consider a more exotic beta-distributed coefficient,
$$
\kappa(x,\omega) = \mathcal{B}^{-1}_{5,2}\left(\frac{1+\mathrm{erf}\left(\frac{\gamma(x,\omega)}{\sqrt{2}}\right)}{2}\right)+1, \qquad \mbox{where} \quad \mathcal{B}_{a,b}(z) = \frac{1}{B(a,b)} \int\limits_{0}^{z} t^{a-1} (1-t)^{b-1} dt.
$$
For the purpose of computing $\phi_i$, the function $\mathcal{B}_{5,2}(z)$ is inverted by the Newton method.
Again, the covariance function is $\cov_{\kappa}(x,y) = \exp\left(-\frac{\|x-y\|^2}{2 l_c^2} \right)$.

Since this distribution varies stronger than the log-normal one, for the computation of the probability \eqref{eq:rare_prob} we use larger $\tau=1.4$.
All other parameters are the same as in the experiments with the log-normal coefficient.
The performance of both TT and Sparse approach in case of the beta distribution is shown in Tables \ref{tab:beta-p}, \ref{tab:beta-M}, \ref{tab:beta-l},
\ref{tab:beta-R_l}, \ref{tab:beta-mc} for $p$, $M$, $l_c$, the spatial grid level and
the Monte Carlo tests, respectively.

We see the same behavior as in the log-normal case.
The only significant difference is the lower error of the Sparse method in the case $M=10$, $R_l=1$, which is 1.08e-3 for the beta distribution and 1.45e-1 for the log-normal one.

\begin{table}[t]
\centering
 \caption{Performance versus $p$, beta distribution}
 \label{tab:beta-p}
 \begin{tabular}{c|ccc|ccc|cc|cc|c}
  $p$ & \multicolumn{3}{c|}{CPU time, sec.} & $r_{\kappa}$ & $r_u$ & $r_{\hat \chi}$ & \multicolumn{2}{c|}{$E_{\kappa}$} & \multicolumn{2}{c|}{$E_{u}$} & $\mathbb{P}$ \\ \hline
      & TT       & Sparse & $\hat\chi$       &              &       &            &  TT & Sparse & TT & Sparse & TT  \\ \hline
  1   & 21.40  & 1.382  & 0.059  & 64   &   49  &     1  &  2.24e-3 &  5.13e-2 &    1.14e-2 &  2.50e-2 & 0       \\
  2   & 39.87  & 5.301  & 0.100  & 65   &   50  &     1  &  1.92e-4 &  5.50e-3 &    7.67e-4 &  1.28e-3 & 0       \\
  3   & 57.16  & 1000   & 70.98  & 65   &   50  &   445  &  9.07e-5 &  1.76e-3 &    5.01e-4 &  1.06e-3 & 1.88e-4     \\
  4   & 76.22  &  ---   & 21.18  & 65   &   50  &   416  &  8.81e-5 &      --- &    1.41e-4 &     ---  & 9.84e-5     \\
  5   & 100.6  &  ---   & 119.7  & 65   &   50  &   428  &  8.89e-5 &      --- &    1.10e-4 &     ---  & 1.23e-4     \\
 \end{tabular}
\end{table}

\begin{table}[t]
\centering
 \caption{Performance versus $M$, beta distribution}
 \label{tab:beta-M}
 \begin{tabular}{c|ccc|ccc|cc|cc|c}
  $M$ & \multicolumn{3}{c|}{CPU time, sec.} & $r_{\kappa}$ & $r_u$ & $r_{\hat \chi}$ & \multicolumn{2}{c|}{$E_{\kappa}$} & \multicolumn{2}{c|}{$E_{u}$} & $\mathbb{P}$ \\ \hline
       & TT     & Sparse & $\hat\chi$       &              &       &            &  TT & Sparse & TT & Sparse & TT  \\ \hline
  10   & 9.777  & 5.796  & 0.942     &     34  &  40 &   39 & 1.70e-4 &   1.65e-3  & 5.18e-4 &  1.08e-3 & 1.95e-4       \\
  15   & 26.46  & 90.34  & 25.16     &     50  &  48 &  374 & 1.03e-4 &   1.73e-3  & 4.96e-4 &  1.08e-3 & 1.94e-4       \\
  20   & 56.92  & 986.2  & 59.57     &     65  &  50 &  413 & 9.15e-5 &   1.80e-3  & 5.08e-4 &  1.09e-3 & 1.88e-4   \\
  30   & 156.7  & 55859  & 147.9     &     92  &  50 &  452 & 7.75e-5 &   7.01e-2  & 5.12e-4 &  4.53e-2 & 1.85e-4
 \end{tabular}
\end{table}

\begin{table}[t]
\centering
 \caption{Performance versus $l_c$, beta distribution}
 \label{tab:beta-l}
 \begin{tabular}{c|ccc|ccc|cc|cc|c}
  $l_c$ & \multicolumn{3}{c|}{CPU time, sec.} & $r_{\kappa}$ & $r_u$ & $r_{\hat \chi}$ & \multicolumn{2}{c|}{$E_{\kappa}$} & \multicolumn{2}{c|}{$E_{u}$} & $\mathbb{P}$ \\ \hline
       & TT      & Sparse & $\hat\chi$       &              &       &            &  TT & Sparse & TT & Sparse & TT  \\ \hline
  0.1   & 665.8  & 55923  & 0.91  &  90 &   48  &   1 & 8.7e-3 &  8.77e-3 &  7.9e-3  &  7.92e-3 &  0  \\
  0.3   & 2983   & 53783  & 1.49  & 177 &   74  &   1 & 1.5e-3 &  2.02e-3 &  1.2e-3  &  1.30e-3 &  0  \\
  0.5   & 1138   & 54297  & 100  & 132 &   74  & 403 & 1.5e-4 &  1.71e-3 &  2.9e-4  &  8.21e-4 &  2.47e-23  \\
  1.0   & 158.8  & 56545  & 153  &  92 &   50  & 463 & 7.8e-5 &  6.92e-2 &  5.1e-4  &  4.47e-2 &  1.96e-04  \\
  1.5   & 62.20  & 55848  & 89.5  &  75 &   42  & 409 & 6.9e-5 &  7.85e-2 &  8.3e-4  &  4.56e-2 &  2.20e-03  \\
 \end{tabular}
\end{table}

\begin{table}[t]
\centering
 \caption{Performance versus $R$, beta distribution}
 \label{tab:beta-R_l}
 \begin{tabular}{c|ccc|ccc|cc|cc|c}
  $\#$DoF & \multicolumn{3}{c|}{CPU time, sec} & $r_{\kappa}$ & $r_u$ & $r_{\hat \chi}$ & \multicolumn{2}{c|}{$E_{\kappa}$} & \multicolumn{2}{c|}{$E_{u}$} & $\mathbb{P}$ \\ \hline
      & TT     & Sparse & $\hat\chi$       &              &       &            &  TT & Sparse & TT & Sparse & TT  \\ \hline
  557   & 9.73  & 5.94  & 0.94   &  34 & 40 & 39   &   1.70e-4 &  1.65e-3 &  5.21e-4 &  1.08e-3  & 1.95e-4       \\
  2145  & 36.2  & 12.7  & 0.77   &  34 & 41 & 41   &   1.56e-4 &  1.64e-3 &  5.19e-4 &  1.08e-3  & 1.97e-4       \\
  8417  & 378  & 162  & 1.12   &  34 & 40 & 43   &   1.53e-4 &  1.62e-3 &  5.07e-4 &  1.06e-3  & 1.96e-4   \\
 \end{tabular}
\end{table}

\begin{table}[t]
\centering
 \caption{Verification of the Monte Carlo method, beta distribution}
 \label{tab:beta-mc}
 \begin{tabular}{c|l|cc|cc||c|c}
  $N_{mc}$ &  $T_{MC}$, sec.		& $E_{\bar u}$ & $E_{\var_u}$  & $\mathbb{P}_*$ & $E_{\mathbb{P}}$ & \multicolumn{2}{c}{TT results} \\  \hline
  $10^2$   &  0.9428                    & 9.12e-3   & 1.65e-1     & 0               & $\infty$ & $T_{solve}$ &  278.4014 sec. \\
  $10^3$   &  9.5612                    & 1.04e-3   & 6.04e-2     & 0               & $\infty$ & $T_{\hat\chi}$            &  179.4764 sec. \\
  $10^4$   &  8.849$\cdot 10^1$         & 4.38e-4   & 5.56e-3     & 0               & $\infty$ & $r_{\kappa}$              &  92           \\
  $10^5$   &  8.870$\cdot 10^2$         & 2.49e-4   & 3.06e-3     & 7.00e-5         & 6.80e-1  & $r_{u}$                   &  50     \\
  $10^6$   &  8.883$\cdot 10^3$         & 8.16e-5   & 8.56e-4     & 1.07e-4         & 9.94e-2  & $r_{\hat\chi}$  	      &  406    \\
           &                            &              &                &                   &              & $\mathbb{P}$             &  1.1765e-04
 \end{tabular}
\end{table}


\section{Conclusion}

We have developed the new block TT cross algorithm to compute the TT approximation of the polynomial chaos expansion of a random field with the tensor product set of polynomials, where the polynomial degrees are bounded individually for each random variable.
The random field can be given as a transformation of a Gaussian field by an arbitrary smooth function.
The new algorithm builds the TT approximation of the PCE in a black box manner.
Compared to the previously existing cross methods, the new approach assimilates all KLE terms simultaneously, which reduces the computational cost significantly.

The uncertain (diffusion) coefficient in the elliptic PDE is approximated via PCE.
We show that the tensor product polynomial set allows a very efficient construction of the stochastic Galerkin matrix in the TT format, provided the coefficient is precomputed in the TT format.
Interestingly, we can even compute the Galerkin matrix exactly by preparing the coefficient with two times larger polynomial orders than those employed for the solution.
In the TT format, we can store the Galerkin matrix in the dense form, since the number of the TT elements $\mathcal{O}(M p^2 r^2)$ is feasible for $p \sim 10$.
This also means that other polynomial families, such as the Chebyshev or Laguerre, may be used straightforwardly.

The Galerkin matrix defines a large linear system on the PCE coefficients of the solution of the stochastic equation.
We solve this system in the TT format via the alternating minimal energy algorithm and calculate the post-processing of the solution, such as the mean, variance and exceedance probabilities.

We demonstrate that with the new TT approach we can go to a larger number of random variables (e.g. $M=30$) used in the PCE (larger stochastic dimension)
and take a larger order of the polynomial approximation in the stochastic space ($p=5$) on a usual PC desktop computer.
For all stages of numerical experiments (computation of the coefficient, operator, solution and statistical functionals)
we report the computational times and the storage costs (TT ranks), and show that they stay moderate in the investigated range of parameters.

In particular, the TT ranks do not grow with the polynomial degree $p$.
This remains in sharp contrast to the traditional sparse polynomial approximation, where the total polynomial degree is bounded.
The cardinality of this sparse polynomial set grows exponentially with $p$, but the tensor product decomposition is not possible anymore.
This renders the total computational cost of the sparse PCE approach higher than the cost of the TT method.
Besides, the tensor product PCE is more accurate than the expansion in the sparse set due to a larger total polynomial degree.
Comparison with the classical Monte Carlo method shows that the TT methods can compute the exceedance probabilities more accurately,
since the TT format approximates the whole stochastic solution implicitly.

Several directions of research can be pursued in the future.

Currently, we store both the matrix and the inverted mean-field preconditioner in the dense form.
This imposes rather severe restrictions on the spatial discretization.
The spatial part of the Galerkin matrix must be dealt with in a more efficient way.

With the tensor product methods the stochastic collocation approach seems very attractive \cite{khos-pde-2010}.
We may introduce quite large discretization grids in each random variable $\theta_m$: additional data compression can be achieved with the QTT approximation \cite{khor-qtt-2011}.
It is important that the deterministic problems are decoupled in the stochastic collocation.
The cross algorithms can become an efficient non-intrusive approach to stochastic equations.

\subsection*{Acknowledgments}
We would like to thank Elmar Zander for his assistance and help in the usage of the Stochastic Galerkin library \emph{sglib}.

A. Litvinenko was supported by the SRI-UQ Center at King Abdullah University of Science and Technology.
A part of this work was done during his stay at Technische Universit\"at Braunschweig and was supported by the German DFG Project CODECS "Effective approaches and solution techniques for conditioning, robust design and control in the subsurface".

S. Dolgov was partially supported by RFBR grants  13-01-12061-ofi-m, 14-01-00804-A, RSCF grants 14-11-00806, 14-11-00659, and the Stipend of the President of Russia during his stay at the Institute of Numerical Mathematics of the Russian Academy of Sciences.


\begin{thebibliography}{10}

\bibitem{Babushka_Colloc}
{\sc I.~Babu\v{s}ka, F.~Nobile, and R.~Tempone}, {\em A stochastic collocation
  method for elliptic partial differential equations with random input data},
  SIAM J. Numer. Anal., 45 (2007), pp.~1005--1034.

\bibitem{babuska2004galerkin}
{\sc I.~Babu\v{s}ka, R.~Tempone, and G.E. Zouraris}, {\em Galerkin finite
  element approximations of stochastic elliptic partial differential
  equations}, SIAM Journal on Numerical Analysis, 42 (2004), pp.~800--825.

\bibitem{back2011stochastic}
{\sc J.~B{\"a}ck, F.~Nobile, L.~Tamellini, and R.~Tempone}, {\em Stochastic
  spectral {G}alerkin and collocation methods for pdes with random
  coefficients: a numerical comparison}, in Spectral and High Order Methods for
  Partial Differential Equations, Springer, 2011, pp.~43--62.

\bibitem{Ballan_Grasedyck14}
{\sc J.~Ballani and L.~Grasedyck}, {\em Hierarchical tensor approximation of
  output quantities of parameter-dependent pdes}, Tech. Report 385, Institut
  f\"ur Geometrie und Praktische Mathematik, RWTH Aachen, Germany, March 2014.

\bibitem{Ballan_Grasedyck10}
{\sc J.~Ballani, L.~Grasedyck, and M.~Kluge}, {\em Black box approximation of
  tensors in hierarchical {T}ucker format}, Linear Algebra and its
  Applications, 438 (2013), pp.~639 -- 657.
\newblock Tensors and Multilinear Algebra.

\bibitem{bebe-2000}
{\sc M.~Bebendorf}, {\em Approximation of boundary element matrices}, Numer.
  Mathem., 86 (2000), pp.~565--589.

\bibitem{bebe-aca-2011}
\leavevmode\vrule height 2pt depth -1.6pt width 23pt, {\em Adaptive cross
  approximation of multivariate functions}, Constructive approximation, 34
  (2011), pp.~149--179.

\bibitem{beylkin-high-2005}
{\sc G.~Beylkin and M.~J. Mohlenkamp}, {\em {Algorithms for numerical analysis
  in high dimensions}}, SIAM J. Sci. Comput., 26 (2005), pp.~2133--2159.

\bibitem{Sudret_sparsePCE}
{\sc G.~Blatman and B.~Sudret}, {\em An adaptive algorithm to build up sparse
  polynomial chaos expansions for stochastic finite element analysis},
  Probabilistic Engineering Mechanics, 25 (2010), pp.~183 -- 197.

\bibitem{Griebel_Bungartz}
{\sc H.-J. Bungartz and M.~Griebel}, {\em Sparse grids}, Acta Numer., 13
  (2004), pp.~147--269.

\bibitem{chkifa-adapt-stochfem-2015}
{\sc A.~Chkifa, A.~Cohen, and Ch. Schwab}, {\em Breaking the curse of
  dimensionality in sparse polynomial approximation of parametric {PDEs}},
  Journal de Mathématiques Pures et Appliquées, 103 (2015), pp.~400 -- 428.

\bibitem{tamen}
{\sc S.~Dolgov}, {\em {tAMEn}}.
\newblock https://github.com/dolgov/tamen.

\bibitem{dkos-eigb-2014}
{\sc S.~V. Dolgov, B.~N. Khoromskij, I.~V. Oseledets, and D.~V. Savostyanov},
  {\em Computation of extreme eigenvalues in higher dimensions using block
  tensor train format}, Computer Phys. Comm., 185 (2014), pp.~1207--1216.

\bibitem{ds-amen-2014}
{\sc S.~V. Dolgov and D.~V. Savostyanov}, {\em Alternating minimal energy
  methods for linear systems in higher dimensions}, SIAM J Sci. Comput., 36
  (2014), pp.~A2248--A2271.

\bibitem{doostan-non-intrusive-2013}
{\sc A.~Doostan, A.~Validi, and G.~Iaccarino}, {\em Non-intrusive low-rank
  separated approximation of high-dimensional stochastic models}, Comput.
  Methods Appl. Mech. Engrg.,  (2013), pp.~42--55.

\bibitem{eigel-adapt-stoch-fem-2014}
{\sc M.~Eigel, C.~J. Gittelson, Ch. Schwab, and E.~Zander}, {\em Adaptive
  stochastic {G}alerkin {FEM}}, Comp. Methods Appl. Mech. Eng., 270 (2014),
  pp.~247 -- 269.

\bibitem{ErnstU10}
{\sc O.~G. Ernst and E.~Ullmann}, {\em Stochastic {G}alerkin matrices}, {SIAM}
  J. Matrix Analysis Applications, 31 (2010), pp.~1848--1872.

\bibitem{wahnert-stochgalerkin-2014}
{\sc M.~Espig, W.~Hackbusch, A.~Litvinenko, H.~G. Matthies, and Ph.
  W{\"a}hnert}, {\em Efficient low-rank approximation of the stochastic
  {G}alerkin matrix in tensor formats}, Computers and Mathematics with
  Applications, 67 (2014), pp.~818--829.

\bibitem{litvinenko-spde-2013}
{\sc M.~Espig, W.~Hackbusch, A.~Litvinenko, H.~G. Matthies, and E.~Zander},
  {\em Efficient analysis of high dimensional data in tensor formats}, in
  Sparse Grids and Applications, Springer, 2013, pp.~31--56.

\bibitem{fannes-mps-1992}
{\sc M.~Fannes, B.~Nachtergaele, and R.F. Werner}, {\em Finitely correlated
  states on quantum spin chains}, Communications in Mathematical Physics, 144
  (1992), pp.~443--490.

\bibitem{Sarkis09}
{\sc J.~Galvis and M.~Sarkis}, {\em Approximating infinity-dimensional
  stochastic darcy's equations without uniform ellipticity}, SIAM Journal on
  Numerical Analysis, 47 (2009), pp.~3624--3651.

\bibitem{matthies-to-be-intrusive-1-2014}
{\sc L.~Giraldi, A.~Litvinenko, D.~Liu, H.~G. Matthies, and A.~Nouy}, {\em To
  be or not to be intrusive? {T}he solution of parametric and stochastic
  equations---the ``plain vanilla'' {G}alerkin case}, SIAM Journal on
  Scientific Computing, 36 (2014), pp.~A2720--A2744.

\bibitem{GITTELSON10}
{\sc C.~J. Gittelson}, {\em Stochastic {G}alerkin discretization of the
  log-normal isotropic diffusion problem}, Mathematical Models and Methods in
  Applied Sciences, 20 (2010), pp.~237--263.

\bibitem{gostz-maxvol-2010}
{\sc S.~A. Goreinov, I.~V. Oseledets, D.~V. Savostyanov, E.~E. Tyrtyshnikov,
  and N.~L. Zamarashkin}, {\em How to find a good submatrix}, in Matrix
  Methods: Theory, Algorithms, Applications, V.~Olshevsky and E.~Tyrtyshnikov,
  eds., World Scientific, Hackensack, NY, 2010, pp.~247--256.

\bibitem{gt-maxvol-2001}
{\sc S.~A. Goreinov and E.~E. Tyrtyshnikov}, {\em The maximal-volume concept in
  approximation by low-rank matrices}, Contemporary Mathematics, 208 (2001),
  pp.~47--51.

\bibitem{graham-QMC-2011}
{\sc I.~G. Graham, F.~Y. Kuo, D.~Nuyens, R.~Scheichl, and I.H. Sloan}, {\em
  {Q}uasi-{M}onte {C}arlo methods for elliptic {PDEs} with random coefficients
  and applications}, J. Comput. Phys., 230 (2011), pp.~3668--3694.

\bibitem{Grasedyck_Kressner}
{\sc L.~Grasedyck, D.~Kressner, and C.~Tobler}, {\em A literature survey of
  low-rank tensor approximation techniques}, GAMM-Mitteilungen, 36 (2013),
  pp.~53--78.

\bibitem{Griebel}
{\sc M.~Griebel}, {\em Sparse grids and related approximation schemes for
  higher dimensional problems}, in Foundations of computational mathematics,
  {S}antander 2005, vol.~331 of London Math. Soc. Lecture Note Ser., Cambridge
  Univ. Press, Cambridge, 2006, pp.~106--161.

\bibitem{KeeseDiss}
{\sc A.~Keese}, {\em Numerical solution of systems with stochastic
  uncertainties. {A} general purpose framework for stochastic finite elements},
  Ph.D. Thesis, TU Braunschweig, Germany,  (2004).

\bibitem{vekh-lattice-2014}
{\sc V.~Khoromskaia and B.~N. Khoromskij}, {\em Grid-based lattice summation of
  electrostatic potentials by assembled rank-structured tensor approximation},
  Comput. Phys. Commun., 185 (2014), pp.~3162 -- 3174.

\bibitem{khor-qtt-2011}
{\sc B.~N. Khoromskij}, {\em $\mathcal{O}(d \log n)$--{Quantics} approximation
  of {$N$--$d$} tensors in high-dimensional numerical modeling}, Constr. Appr.,
  34 (2011), pp.~257--280.

\bibitem{Hcovariance}
{\sc B.~N. Khoromskij, A.~Litvinenko, and H.~G. Matthies}, {\em Application of
  hierarchical matrices for computing the {K}arhunen-{L}o{\`e}ve expansion},
  Computing, 84 (2009), pp.~49--67.

\bibitem{khos-pde-2010}
{\sc B.~N. Khoromskij and I.~V. Oseledets}, {\em {Quantics-TT} collocation
  approximation of parameter-dependent and stochastic elliptic {PDEs}}, Comput.
  Meth. Appl. Math., 10 (2010), pp.~376--394.

\bibitem{khos-qtt-2010}
\leavevmode\vrule height 2pt depth -1.6pt width 23pt, {\em {QTT}-approximation
  of elliptic solution operators in higher dimensions}, Rus. J. Numer. Anal.
  Math. Model., 26 (2011), pp.~303--322.

\bibitem{KhSch-Galerkin-SPDE-2011}
{\sc B.~N. Khoromskij and Ch. Schwab}, {\em Tensor-structured {Galerkin}
  approximation of parametric and stochastic elliptic {PDEs}}, SIAM J. of Sci.
  Comp., 33 (2011), pp.~1--25.

\bibitem{KressTobler11}
{\sc D.~Kressner and Ch. Tobler}, {\em Low-rank tensor {K}rylov subspace
  methods for parametrized linear systems}, SIAM J. Matrix Anal. Appl., 32
  (2011), pp.~1288--1316.

\bibitem{KressTobler11HighDim}
\leavevmode\vrule height 2pt depth -1.6pt width 23pt, {\em Preconditioned
  low-rank methods for high-dimensional elliptic {PDE} eigenvalue problems},
  Comput. Methods Appl. Math., 11 (2011), pp.~363--381.

\bibitem{kunoth-2013}
{\sc A.~Kunoth and C.~Schwab}, {\em Analytic regularity and {GPC} approximation
  for control problems constrained by linear parametric elliptic and parabolic
  {PDEs}}, SIAM J Control and Optim., 51 (2013), pp.~2442--2471.

\bibitem{Schwab-MLQMC-2015}
{\sc F.~Y. Kuo, Ch. Schwab, and I.~H. Sloan}, {\em Multi-level quasi-{M}onte
  {C}arlo finite element methods for a class of elliptic pdes with random
  coefficients}, Foundations of Computational Mathematics,  (2015), pp.~1--39.

\bibitem{LeMatre10}
{\sc O.~P. Le~Ma{\^{\i}}tre and O.~M. Knio}, {\em Spectral methods for
  uncertainty quantification}, Series: Scientific Computation, Springer, New
  York, 2010.

\bibitem{litvinenko2009sparse}
{\sc A.~Litvinenko and H.~G. Matthies}, {\em Sparse data representation of
  random fields}, PAMM, 9 (2009), pp.~587--588.

\bibitem{Bieri2009}
{\sc M.~Marcel and Ch. Schwab}, {\em {Sparse high order FEM for elliptic
  sPDEs}}, Computer Methods in Applied Mechanics and Engineering, 198 (2009),
  pp.~1149--1170.

\bibitem{Matthies_encycl}
{\sc H.~G. Matthies}, {\em Uncertainty Quantification with Stochastic Finite
  Elements.}, John Wiley and Sons, Ltd, Chichester West Sussex, 2007.
\newblock Part 1. Fundamentals. Encyclopedia of Computational Mechanics.

\bibitem{matthiesKeese05cmame}
{\sc H.~G. Matthies and A.~Keese}, {\em Galerkin methods for linear and
  nonlinear elliptic stochastic partial differential equations}, Computer
  Methods in Applied Mechanics and Engineering, 194 (2005), pp.~1295--1331.

\bibitem{UQLitvinenko12}
{\sc H.~G. Matthies, A.~Litvinenko, O.~Pajonk, B.~V. Rosi\'c, and E.~Zander},
  {\em Parametric and uncertainty computations with tensor product
  representations}, in Uncertainty Quantification in Scientific Computing,
  Andrew~M. Dienstfrey and Ronald~F. Boisvert, eds., vol.~377 of IFIP Advances
  in Information and Communication Technology, Springer Berlin Heidelberg,
  2012, pp.~139--150.

\bibitem{matthieszander-lowrank-2012}
{\sc H.~G. Matthies and E.~Zander}, {\em Solving stochastic systems with
  low-rank tensor compression}, Linear Algebra and its Applications, 436
  (2012), pp.~3819--3838.

\bibitem{mugler2011elliptic}
{\sc A.~Mugler and H.-J. Starkloff}, {\em On elliptic partial differential
  equations with random coefficients}, Stud. Univ. Babes-Bolyai Math, 56
  (2011), pp.~473--487.

\bibitem{nobile2008sparse}
{\sc F.~Nobile, R.~Tempone, and C.~G. Webster}, {\em A sparse grid stochastic
  collocation method for partial differential equations with random input
  data}, SIAM Journal on Numerical Analysis, 46 (2008), pp.~2309--2345.

\bibitem{Nouy07}
{\sc A.~Nouy}, {\em A generalized spectral decomposition technique to solve a
  class of linear stochastic partial differential equations}, Comput. Methods
  Appl. Mech. Engrg., 196 (2007), pp.~4521--4537.

\bibitem{nouy-pgd-stoch-2010}
\leavevmode\vrule height 2pt depth -1.6pt width 23pt, {\em Proper generalized
  decompositions and separated representations for the numerical solution of
  high dimensional stochastic problems}, Archives of Computational Methods in
  Engineering, 17 (2010), pp.~403--434.

\bibitem{LitvNowak13}
{\sc W.~Nowak and A.~Litvinenko}, {\em Kriging and spatial design accelerated
  by orders of magnitude: Combining low-rank covariance approximations with
  {FFT}-techniques}, Mathematical Geosciences, 45 (2013), pp.~411--435.

\bibitem{osel-tt-2011}
{\sc I.~V. Oseledets}, {\em Tensor-train decomposition}, SIAM J. Sci. Comput.,
  33 (2011), pp.~2295--2317.

\bibitem{tt-toolbox}
{\sc I.~V. Oseledets, S.~Dolgov, V.~Kazeev, D.~Savostyanov, O.~Lebedeva,
  P.~Zhlobich, T.~Mach, and L.~Song}, {\em {TT-Toolbox}}.
\newblock https://github.com/oseledets/TT-Toolbox.

\bibitem{ot-tt-2009}
{\sc I.~V. Oseledets and E.~E. Tyrtyshnikov}, {\em Breaking the curse of
  dimensionality, or how to use {SVD} in many dimensions}, SIAM J. Sci.
  Comput., 31 (2009), pp.~3744--3759.

\bibitem{ot-ttcross-2010}
\leavevmode\vrule height 2pt depth -1.6pt width 23pt, {\em {TT-cross}
  approximation for multidimensional arrays}, Linear Algebra Appl., 432 (2010),
  pp.~70--88.

\bibitem{sav-qott-2014}
{\sc D.~V. Savostyanov}, {\em Quasioptimality of maximum--volume cross
  interpolation of tensors}, Linear Algebra Appl.,  (2014).

\bibitem{so-dmrgi-2011proc}
{\sc D.~V. Savostyanov and I.~V. Oseledets}, {\em Fast adaptive interpolation
  of multi-dimensional arrays in tensor train format}, in Proceedings of 7th
  International Workshop on Multidimensional Systems (nDS), IEEE, 2011.

\bibitem{schollwock-2011}
{\sc U.~Schollw\"ock}, {\em The density-matrix renormalization group in the age
  of matrix product states}, Annals of Physics, 326 (2011), pp.~96--192.

\bibitem{SCHWAB2006}
{\sc Ch. Schwab and R.~A. Todor}, {\em {Karhunen-Lo\`{e}ve approximation of
  random fields by generalized fast multipole methods}}, Journal of
  Computational Physics, 217 (2006), pp.~100--122.

\bibitem{Sousedik2014truncated}
{\sc B.~Soused{\'\i}k and R.~Ghanem}, {\em Truncated hierarchical
  preconditioning for the stochastic {G}alerkin {FEM}}, International Journal
  for Uncertainty Quantification, 4 (2014), pp.~333--348.

\bibitem{scheichl-mlmc-further-2013}
{\sc A.~L. Teckentrup, R.~Scheichl, M.B. Giles, and E.~Ullmann}, {\em Further
  analysis of multilevel {M}onte {C}arlo methods for elliptic {PDEs} with
  random coefficients}, Numer. Math., 125 (2013), pp.~569--600.

\bibitem{TodorSCA}
{\sc R.~A. Todor and Ch. Schwab}, {\em Convergence rates for sparse chaos
  approximations of elliptic problems with stochastic coefficients}, IMA J.
  Numer. Anal., 27 (2007), pp.~232--261.

\bibitem{Ullmann10}
{\sc E.~Ullmann}, {\em A {K}ronecker product preconditioner for stochastic
  {G}alerkin finite element discretizations}, {SIAM} J. Scientific Computing,
  32 (2010), pp.~923--946.

\bibitem{ullmann2012efficient}
{\sc E.~Ullmann, H.~Elman, and O.~G. Ernst}, {\em Efficient iterative solvers
  for stochastic {G}alerkin discretizations of log-transformed random diffusion
  problems}, SIAM Journal on Scientific Computing, 34 (2012), pp.~A659--A682.

\bibitem{white-dmrg-1993}
{\sc S.~R. White}, {\em Density-matrix algorithms for quantum renormalization
  groups}, Phys. Rev. B, 48 (1993), pp.~10345--10356.

\bibitem{wiener38}
{\sc N.~Wiener}, {\em The homogeneous chaos}, American Journal of Mathematics,
  60 (1938), pp.~897--936.

\bibitem{xiuKarniadakis02a}
{\sc D.~Xiu and G.~E. Karniadakis}, {\em The {W}iener-{A}skey polynomial chaos
  for stochastic differential equations}, SIAM J. Sci. Comput., 24 (2002),
  pp.~619--644.

\bibitem{sglib}
{\sc E.~Zander}, {\em Stochastic {G}alerkin library:},
  http://github.com/ezander/sglib,  (2008).

\bibitem{ZanderDiss}
\leavevmode\vrule height 2pt depth -1.6pt width 23pt, {\em Tensor Approximation
  Methods for Stochastic Problems}, PhD thesis, Dissertation, Technische
  Universit\"at Braunschweig, 2013.

\end{thebibliography}
\end{document}